\documentclass[12pt]{amsart} 
\usepackage{amssymb}
\usepackage{verbatim}
\usepackage{hyperref}
\usepackage{color}

\usepackage{tikz}

\begin{document}
\setcounter{tocdepth}{1}

\newtheorem{theorem}{Theorem}    
\newtheorem{proposition}[theorem]{Proposition}
\newtheorem{corollary}[theorem]{Corollary}
\newtheorem{lemma}[theorem]{Lemma}
\newtheorem{sublemma}[theorem]{Sublemma}
\newtheorem{conjecture}[theorem]{Conjecture}
\newtheorem{fact}[theorem]{Fact}
\newtheorem{observation}[theorem]{Observation}

\newtheorem{definition}{Definition}
\newtheorem{notation}[definition]{Notation}
\newtheorem{remark}[definition]{Remark}
\newtheorem{question}[definition]{Question}
\newtheorem{questions}[definition]{Questions}
\newtheorem{hypothesis}[definition]{Hypothesis}

\newtheorem{example}[definition]{Example}
\newtheorem{problem}[definition]{Problem}
\newtheorem{exercise}[definition]{Exercise}

 \numberwithin{theorem}{section}
 \numberwithin{equation}{section}

\def\repair{\medskip\hrule\hrule\medskip}

\def\bff{\mathbf f}
\def\bE{\mathbf E}
\def\bF{\mathbf F}
\def\bK{\mathbf K}
\def\bP{\mathbf P}
\def\bx{\mathbf x}
\def\bt{\mathbf t}
\def\bc{\mathbf c}
\def\ba{\mathbf a}
\def\bw{\mathbf w}
\def\bh{\mathbf h}
\def\bk{\mathbf k}
\def\bn{\mathbf n}
\def\bg{\mathbf g}
\def\bc{\mathbf c}
\def\bs{\mathbf s}
\def\bp{\mathbf p}
\def\by{\mathbf y}
\def\bv{\mathbf v}
\def\be{\mathbf e}
\def\bu{\mathbf u}
\def\bm{\mathbf m}
\def\bxi{{\mathbf \xi}}
\def\bR{\mathbf R}
\def\by{\mathbf y}
\def\bz{\mathbf z}
\def\bfb{\mathbf b}
\def\bbS{\mathbb S}

\def\bPhi{\mathbf \Phi}

\newcommand{\norm}[1]{ \|  #1 \|}

\def\scriptl{{\mathcal L}}
\def\scriptc{{\mathcal C}}
\def\scriptd{{\mathcal D}}
\def\scrapd{{\mathcal D}}
\def\scripts{{\mathcal S}}
\def\scriptq{{\mathcal Q}}
\def\scriptt{{\mathcal T}}
\def\scriptf{{\mathcal F}}
\def\scriptm{{\mathcal M}}
\def\scripti{{\mathcal I}}
\def\scriptr{{\mathcal R}}
\def\scriptb{{\mathcal B}}
\def\scripte{{\mathcal E}}
\def\scripta{{\mathcal A}}
\def\scriptn{{\mathcal N}}
\def\scriptv{{\mathcal V}}
\def\scriptz{{\mathcal Z}}
\def\scriptj{{\mathcal J}}

\def\bk{\mathbf k}
\def\kernel{\operatorname{kernel}}
\def\eps{\varepsilon}

\def\reals{\mathbb R}
\def\naturals{\mathbb N}
\def\integers{\mathbb Z}
\def\rationals{\mathbb Q}
\def\one{\mathbf 1}
\def\complex{{\mathbb C}\/}

\def\lt{{L^2}}

 \def\four{\underline{4}}
 \def\three{\underline{3}}

\def\bart{\bar t}
\def\barz{\bar z}
\def\barx{\bar x}
\def\bary{\bar y}
\def\bars{\bar s}

\def\distance{\operatorname{distance}}
\def\md{{\mathcal D}}
\def\Pphi{{\mathbf \Phi}}
\def\Ppsi{{\mathbf \Psi}}

\title[Implicitly oscillatory quadrilinear integrals]
{On Implicitly Oscillatory Quadrilinear Integrals}

 \author{Michael Christ}

\address{
        Michael Christ\\
        Department of Mathematics\\
        University of California \\
        Berkeley, CA 94720-3840, USA}
\email{mchrist@berkeley.edu}

\date{February 9, 2022.}


\thanks{Research supported by NSF grant
DMS-1901413}

\maketitle


\section{Introduction}

\subsection{Formulation of the main result}
Let $B$ be a nonempty open ball of finite radius in $\reals^2$,
and let $\tilde B$ be a connected open neighborhood of $B$.
For $j\in\{1,2,3,4\}$, let $\varphi_j:\tilde B\to\reals$ be a $C^\omega$ submersion.
Write 
$\three = \{1,2,3\}$, $\four = \{1,2,3,4\}$, and 
$\Pphi = (\varphi_j: j\in\four)$.

Consider quadrilinear functionals
\begin{equation}
\scriptt(\bff) = \int_{\reals^2} \prod_{j\in\four} (f_j\circ\varphi_j)(x)\,\eta(x) \,dx,
\end{equation}
acting on $\bff = (f_j: j\in\four)$, with each
$f_j:\reals^1\to\complex$ Lebesgue measurable and belonging to an appropriate function space.
Throughout the discussion, $\eta\in C^\infty(\reals^2)$ is a smooth cutoff function supported in $B$.

We impose the following hypothesis, which is a necessary condition for the conclusion that we seek.
\begin{hypothesis}[Main hypothesis] \label{mainhypothesis} 
Let $\Omega\subset B$ be open and connected. Let
$(F_j: j\in\four)$ be a tuple of $C^\omega$ functions
$F_j:\varphi_j(\Omega)\to\reals$. 
If
$\sum_{j=1}^4 F_j\circ\varphi_j\equiv 0$  in $\Omega$
then each $F_j$ is constant.
\end{hypothesis}


For $\sigma\in\reals$ and $p\in(1,\infty)$,
denote by $W^{p,\sigma}$ the Sobolev space of functions
having $\sigma$ derivatives in  $L^p(\reals)$.
Our main result is the following theorem, whose other hypotheses are formulated below.

\begin{theorem} \label{maintheorem}
Let $\Pphi$ satisfy the main hypothesis \ref{mainhypothesis},
the transversality hypothesis \ref{transversehyp}, 
and the auxiliary hypotheses~\ref{auxhyp1}, \ref{auxhyp2}, and \ref{auxhyp3}.
For any $p>2$ and any $\eta\in C^\infty_0(B)$,
there exist $\sigma<0$ and $C<\infty$ such that
for every four-tuple $\bff = (f_1,\dots,f_4)$ of functions in $L^p(\reals^1)$,
\begin{equation} \label{maininequality}
|\scriptt(\bff)| \le C \prod_{j\in\four} \norm{f_j}_{W^{p,\sigma}}.
\end{equation}
\end{theorem}

A special case of the corresponding result with two factors $f_j$
is a restatement of this well-known fact: Let $t\mapsto\gamma(t)$
be a real analytic mapping from $\reals^1$ to $\reals^2$.
Let $\eta\in C^\infty(\reals)$ be compactly supported.
Define a measure $\mu$ in $\reals^2$ by $\int g\,d\mu = \int_\reals g(\gamma(t))\,\eta(t)\,dt$.
If the range of $\gamma$ is not contained in any affine subspace of $\reals^2$
then there exists $\delta>0$ such that the Fourier transform of $\mu$
satisfies $|\mu(\xi)| = O(|\xi|^{-\delta})$ as $|\xi|\to\infty$.

An immediate application is to weak continuity of 
mappings $\bff\mapsto \prod_{j\in\four} (f_j\circ\varphi_j)$.
Joly, M\'etivier, and Rauch \cite{JMR} proved weak continuity
with this application in mind,
for threefold products, with weaker auxiliary hypotheses all around. 
Their proof was based on microlocal defect
measures and an argument by contradiction, and yielded no quantitative upper bound
corresponding to \eqref{maininequality}. 
For threefold products, an inequality in terms of negative order Sobolev norms 
was proved by Bourgain \cite{bourgain_nonlinear_roth} in a particular case,
and later by the author \cite{triosc} in a relatively general $C^\omega$ case. 
More recently, Evans \cite{evans} has developed another proof for threefold products.
A more streamlined version of the analysis in \cite{triosc}, with
certain supplementary hypotheses relaxed, is presented in \cite{christzhou}.

Fourfold products, with integration over two-dimensional domains,
are more singular and apparently require significantly more intricate analysis.
One indication of this is the failure of the inequality to hold for $p=2$
with four factors, while it extends to all $p>3/2$ for three factors,
leading to a useful gain in certain Fourier coefficient-based estimates
--- the same phenomenon that arises in Roth's theorem concerning
arithmetic progressions of length three \cite{taovu}.

\subsection{Hypotheses, and comments on them}
Hypothesis \ref{mainhypothesis} is necessary 
for the conclusion \eqref{maininequality} to hold. 
Indeed, suppose $\sum_j (F_j\circ\varphi_j)\equiv 0$ in some open set.
For large $\lambda\in\reals^+$, define $f_j = \eta_j e^{i\lambda F_j}$ 
for suitable $C^\infty_0$ cutoff functions $\eta_j$. 
Then $\prod_j (f_j\circ\varphi_j)$ is independent of $\lambda$,
and the cutoff functions can be chosen so that its integral does not vanish.
On the other hand, if some real-valued function $F_k$ 
is nonconstant, and if $\eta_k$ is chosen to be supported in an open set
in which $\nabla F_k$ vanishes nowhere, 
then for any strictly negative parameter $\sigma$, $\norm{\eta_k e^{i\lambda F_k}}_{W^{p,\sigma}}$
tends to $0$ as $\lambda\to\infty$. Thus \eqref{maininequality} cannot hold.

Throughout the paper, we impose a transversality hypothesis on $\Pphi$.
\begin{hypothesis}[Transversality] \label{transversehyp}
For each pair of distinct indices $j\ne k\in\four$,
$\nabla\varphi_j,\nabla\varphi_k$ are linearly independent at each point of $B$.
\end{hypothesis}

This transversality hypothesis guarantees that the integral defining $\scriptt(\bff)$
converges absolutely when each $f_j\in\lt$, and that 
$|\scriptt(\bff)| \le C\prod_{j=1}^4 \norm{f_j}_{\lt}$.
In this inequality,
$L^2$ cannot be replaced by $L^p$ for any $p<2$.
This can be seen by choosing any point $x$, setting $f_j$
equal to the indicator function of an interval of length $\eps$
centered at $\varphi_j(x)$, and evaluating both sides of
the inequality as $\eps\to 0$.

We will impose three auxiliary hypotheses, each of which holds for generic data $\Pphi$.
For each $j\in\four$, let $V_j$ be a nowhere vanishing $C^\omega$
vector field in $\tilde B$ that satisfies $V_j(\varphi_j)\equiv 0$.

\begin{hypothesis} \label{auxhyp1}
For any $k\in\four$,
on any connected open set $\omega\subset \tilde B \subset \reals^2$ with nonempty interior,
any $C^\omega$ solution $(F_j: j\ne k)$ of the equation
\begin{equation} \label{eq:auxhyp1}
\sum_{j\ne k} (F_j\circ\varphi_j)\, V_k\varphi_j\, \nabla\varphi_j \equiv 0
\end{equation}
in $\omega$ satisfies $F_j\circ\varphi_j \equiv 0$ in $\omega$ for each $j\ne k$.
\end{hypothesis}

Equation \eqref{eq:auxhyp1} is a system of two scalar equations
in three scalar-valued unknown functions $F_j$.

\begin{hypothesis} \label{auxhyp2}
Endow $\reals^3 = \reals^2\times\reals^1$ with coordinates
$(x,s)$. For each $k\in\four$ and each $j\ne k$, define 
\begin{equation} \Phi_{j,k} = (\varphi_j(x),sV_k\varphi_j(x)).\end{equation}
Define $W_{j,k}$ to be nowhere vanishing $C^\omega$ vector
fields in $\tilde B \times\reals^1$ that satisfy $W_{j,k}(\Phi_{j,k})\equiv 0$.
Then for each $k\in\four$, the three-tuple of vector fields $\{W_{j,k}: j\ne k\}$
is linearly independent at each point of $\tilde B\times (\reals\setminus\{0\})$. 
\end{hypothesis}

If $i,j,k\in\four$ are pairwise distinct then since 
$\nabla_x\varphi_i,\nabla\varphi_j$ are linearly independent
at each point of $\tilde B$,
then $\nabla_{x,s}\Phi_{i,k}$ and $\nabla_{x,s}\Phi_{j,k}$
are likewise linearly independent at each point of $\tilde B \times \reals$.

\begin{hypothesis} \label{auxhyp3}
For each permutation $(i,j,k,l)$ of $(1,2,3,4)$, for each $\tau\in\reals$, the function
\begin{equation}
\frac{V_l\varphi_i \cdot |V_k\varphi|^\tau}
{V_l\varphi_j \cdot |V_k\varphi_j|^\tau}
\end{equation}
cannot be expressed as a ratio $\frac{h_i\circ\varphi_i}{h_j\circ\varphi_j}$,
with $h_i,h_j$ real analytic, in any nonempty open subset of $\reals^2$.
\end{hypothesis}

It can be shown 
that generic tuples $\Pphi$
do satisfy the main hypothesis \ref{mainhypothesis},
and more generally,  generic $n$-tuples $\Pphi = (\varphi_1,\dots,\varphi_n)$ of mappings 
satisfy the analogous condition for arbitrarily large $n$.
More precisely, there exist $N<\infty$ and a nonempty Zariski open subset $\Omega$ of
the Euclidean space of all $N$-jets of $n$-tuples of functions at $0\in\reals^2$
such that for any real analytic $\Pphi$ whose $N$-jet at $0$ belongs to $\Omega$,
there exists a ball $B$ centered at $0$ in which hypothesis \ref{mainhypothesis} is satisfied.
Likewise, all of the auxiliary hypotheses also hold for generic $\Pphi$.

The main hypothesis implies a weak form of transversality:
for any $j\ne k$, $V_k(\varphi_j)$ does not vanish identically.

The hypothesis of real analyticity of $\Pphi$ represents a compromise that simplifies
both the formulation of Theorem~\ref{maintheorem}, and its proof. As is shown in \cite{sublevel4}
and in \cite{christzhou},
it implies a formally stronger quantitative formulation:
There exist $C,N$ such that for all $\bff\in C^N$ and all $x\in B$,
\begin{equation} \label{relaxmainhyp}
\sum_{j=1}^4 \sum_{0<n\le N} \big| \partial^n(f_j\circ\varphi_j)(x)\big|
\le C \sum_{0<|\alpha|\le N} \big|\,\partial^\alpha \sum_{k=1}^4 (f_k\circ\varphi_k)(x)\, \big|.
\end{equation}
Here $\partial^n = d^n/dy^n$ and $\partial^\alpha = \partial^\alpha/\partial x^\alpha$.
The author believes that for any $N$, 
Theorem~\ref{maintheorem} extends to $C^{N+1}$ data $\Pphi$ satisfying \eqref{relaxmainhyp},
with the transversality and auxiliary hypotheses significantly relaxed.
For real analytic $\Pphi$, we believe that the transversality hypothesis
can be relaxed to the assumption that no two $\nabla\varphi_j$
are everywhere linearly dependent, and that the three auxiliary hypotheses can be omitted entirely.
It is shown for threefold products in \cite{christzhou}
and for a related additive problem in \cite{sublevel4} 
how results with relaxed auxiliary hypotheses can be obtained as consequences
of results with strong auxiliary hypotheses, in closely related contexts.



\medskip
Key ingredients in the proof of Theorem~\ref{maintheorem} were developed 
in \cite{triosc}, \cite{CDR}, and \cite{sublevel4} in that order,
with their eventual application to Theorem~\ref{maintheorem} envisioned.

The author is indebted to Terence Tao and Craig Evans for calling to his attention
the works of Bourgain \cite{bourgain_nonlinear_roth}
and of Joly, M\'etivier, and Rauch \cite{JMR}, respectively.

\section{Parameters}

Certain positive parameters are to be chosen in the course of the proof.
In order to avoid any appearance of circularity, we specify here the order in which these
are to be chosen, which is not the order in which they will arise in the exposition.
$\gamma\in (\tfrac12,1)$ is to be chosen first, and is an arbitrary element of that
interval. Each of the other parameters is
chosen to be sufficiently small relative to parameters chosen earlier, and relative to
other quantities that are determined, at least implicitly, by parameters chosen earlier.
The next parameter chosen is $\tau_0$,
then $\rho_0,\delta_0,\delta_1,\delta_3,\delta_4,\delta_1^*,\delta_2^*,\delta_3^*,\delta_4^*$,
and finally $\rho$, in that order.
Thus at every step, we may assume that $\rho$ is as small as may be desired,
relative to all other positive parameters that arise.
$\delta_2$ and $\delta_5$ are not chosen, but are determined by other parameters.

\section{Decompositions}

\subsection{Fourier decomposition}
Under the transversality hypothesis, one has
for any permutation $(i,j,k,l)$ of $(1,2,3,4)$
\begin{equation}
|\scriptt(\bff)| \le C \norm{f_i}_{L^1}\norm{f_j}_{L^1}
\norm{f_k}_{L^\infty}\norm{f_l}_{L^\infty}.
\end{equation}
Therefore by interpolation,
\begin{equation}
|\scriptt(\bff)| \le C \prod_{j\in\four} \norm{f_j}_{L^2}.
\end{equation}
Therefore Theorem~\ref{maintheorem} is an immediate consequence,
by decomposition and interpolation, of the following proposition.

\begin{proposition} \label{mainprop}
Let $\Pphi$ satisfy the hypotheses of Theorem~\ref{maintheorem}.
There exists $\delta>0$ such that 
for any $\eta\in C^\infty_0(\reals^2)$ supported in $B$,
there exists $C<\infty$ with the following property.
Let $\lambda\ge 1$. Let $f_j\in L^\infty(\reals)$. 
Assume that $\widehat{f_j(\xi)}$ is supported where $|\xi| \le 2\lambda$
for each $j\in\four$,
and that $\widehat{f_1(\xi)}$ is supported where $\tfrac12\lambda \le |\xi| \le 2\lambda$.  Then
\begin{equation} \label{ineq:mainprop}
|\scriptt(\bff)| \le C \lambda^{-\delta} \prod_{j\in\four} \norm{f_j}_{L^\infty}.
\end{equation}
\end{proposition}

Throughout the analysis, we assume without loss of generality that $\lambda$ is large.
Thus if $c_1<c_2$ are constants then $\lambda^{c_1}$ is small relative to $\lambda^{c_2}$.

\subsection{$\flat/\sharp$ decomposition}
A basic strategy is to apply the Cauchy-Schwarz inequality to eliminate
one of the four functions $f_j$. 
This transforms the remaining functions $f_k$ to associated functions $\scriptd_sf_k$, as follows.

\begin{definition}
For $x,s\in\reals$ and $f:\reals\to\complex$, 
\begin{equation} \scriptd_s f(x) = f(x+s)\overline{f(x)}.\end{equation}
\end{definition}

If $f$ is highly oscillatory, then $\scriptd_s f$ need not be so;
if for instance $f(x)=e^{i\lambda x}$ then $\scriptd_s f$
is a constant function of $x$ for each $s$.
On the other hand, if $f(x)=e^{i\lambda x^2}$ with $\lambda\in\reals$ and $|\lambda|$ large,
then $\scriptd_s f(x)$ takes the form $c(s)e^{i2s\lambda x}$ with $|c(s)|\equiv 1$.
The next lemma is based on a simple characterization of those
functions for which $\scriptd_s f$ is not highly oscillatory.

\begin{lemma} \cite{CDR} \label{lemma:sharpflat}
Let $\delta>0$ and $R\ge 1$.
For any $f\in\lt(\reals)$ there exists a decomposition $f = f_\sharp + f_\flat$
satisfying $\norm{f_\sharp}_\lt+\norm{f_\flat}_\lt \lesssim  \norm{f}_\lt$,
with the following supplementary properties.

The summand $f_\sharp$ admits a decomposition 
\begin{equation}
f_\sharp(x) = \sum_{n=1}^{M} h_n(x) e^{i \alpha_{n}x}
\end{equation}
with each  $\alpha_{n}\in\reals$, 
and with each $h_n$ a smooth function satisfying 
\begin{equation} \left\{
\begin{gathered}
\|\partial^N h_n\|_{\infty}
\le C_N R^{N} \|f\|_\infty\ \forall\, N\ge 0,
\\
\norm{h_{n}}_{L^2} \lesssim \norm{f}_{\lt}, 
\\
\widehat{h_n} \text{ is supported in $[-R,R]$,}
\\
{M}\lesssim R^{\delta}. 
\end{gathered} \right. \end{equation}
Moreover, 
the support of $\widehat{f_\sharp}$ is contained in 
a $CR$--neighborhood of the support of $\widehat{f}$.
Finally, if $f\in L^\infty$ then
\begin{equation} \label{extraLinftybound} \norm{f_\sharp}_\infty 
\lesssim M^{1/2} \norm{f}_\infty. \end{equation}

The summand $f_\flat$ satisfies
\begin{equation} \label{gflat}
\int_{\reals} \int_{|\xi| \le R}
|\widehat{\md_s f_\flat}(\xi)|^2\,d\xi\,ds
\lesssim R^{-\delta}\norm{f}_{\lt}^4.
\end{equation}
All implicit constants are independent of $R,f$.
\end{lemma}

Thus for most $s$, the function $x\mapsto \scriptd_s f_\flat(x)$ either has small norm, 
or is highly oscillatory.
The conclusion \eqref{extraLinftybound} is not stated explicitly in \cite{CDR},
but is an immediate consequence of the construction given there.

\subsection{Local $\flat/\sharp$ decomposition}\label{subsection:Local_decomp}
The proof of Theorem~\ref{maintheorem} is structured around a local version of Lemma~\ref{lemma:sharpflat}.
Fix auxiliary parameters $\gamma\in(\tfrac12,1)$ and $\tau_0>0$ satisfying
\begin{equation} \gamma+\tau_0<1.\end{equation}

Let $\bff$ be given, satisfying the hypotheses of Proposition~\ref{mainprop}
with respect to a large parameter $\lambda\in\reals^+$.
We may assume henceforth, with no loss of generality, 
that $\norm{f_j}_\infty = O(1)$ for each $j\in\four$.

For $m\in\integers$ let $I_m,I_m^*$ be the intervals of lengths
$\lambda^{-\gamma}$ and $3\lambda^{-\gamma}$, respectively,
centered at $m\lambda^{-\gamma}$.
Introduce a partition of unity $1 = \sum_m \eta_m^2$
in $\reals^1$, with each $\eta_m\in C^\infty_0$ supported
in the interval concentric with $I_m$ of sidelength $2\lambda^{-\gamma}$,
and satisfying $\norm{\nabla^N \eta_m}_{C^0} = 
O(\lambda^{N\gamma})$ for every nonnegative integer $N$.
Decompose
\begin{equation} f_j = \sum_{m\in\integers} \eta_m f_{j,m} 
\ \text{ with } \ f_{j,m} = \eta_{m} f_j. \end{equation}

For each $j\in\four$ and $m\in\integers$,
apply Lemma~\ref{lemma:sharpflat} to the function 
$y\mapsto f_{j,m}(\lambda^{-\gamma} y)$, 
with the parameter $R$ of that lemma defined to be $R = \lambda^{\tau_0}$,
and with the parameter $\delta$ of that lemma chosen to be a small number 
$\delta_j^*>0$ that depends on $j$.
In the final step of the proof of Theorem~\ref{maintheorem}, in \S\ref{section:final_conclusion},
the four values $\delta_j^*$ will be chosen, 
but for the present, each $\delta_j^*$ is small but otherwise arbitrary.

Reverse this linear change of variables for each $m$, to express
\begin{equation} 
f_j = f_{j,\sharp}+f_{j,\flat}\ \text{ with } \ 
\left\{
\begin{aligned}
	f_{j,\sharp} &= \sum_{m\in\integers} \eta_m \cdot f_{j,m,\sharp}
\\
	f_{j,\flat} &= \sum_{m\in\integers} \eta_m \cdot f_{j,m,\flat}
\end{aligned} \right. \end{equation}
where the terms $f_{j,m,\sharp}$ enjoy these properties:
\begin{equation} \label{fsharpsums}
f_{j,m,\sharp}(x) = \sum_{n=1}^M h_{j,m,n}(x) e^{i\alpha_{j,m,n}x}
\end{equation}
with 
\begin{equation} \label{fsharpproperties}\left\{
\begin{gathered}
|\alpha_{j,m,n}| = O(\lambda),
\\
M = O(\lambda^{\delta_j^* \tau_0/2}),
\\
\norm{\partial^N h_{j,m,n}}_\infty = O_N(\lambda^{(\tau_0+\gamma) N})
\ \text{ 
for every $N\in\naturals$,}
\\
\text{$\widehat{h_{j,m,n}}(\xi)$ is supported
where $|\xi|\le \lambda^{\tau_0+\gamma}$.}
\end{gathered} \right. \end{equation} 
Moreover, 
\begin{equation} \label{alpha1lowerbound}
|\alpha_{1,m,n}| \gtrsim \lambda.
\end{equation}
This lower bound for $|\alpha_{1,m,n}|$ holds by the 
assertion concerning the support of $\widehat{f_\sharp}$
in Lemma~\ref{lemma:sharpflat}, since
$\widehat{f_1}(\xi)$ is assumed to be supported where $|\xi|\gtrsim\lambda$
and $R = \lambda^{\tau_0}$ is chosen to be small relative to $\lambda$.

The summands $f_{j,m,\flat}$ satisfy
\begin{equation} \label{keyflatproperty}
\int_{\reals} \int_{|\xi| \le \lambda^{\tau_0+\gamma}} 
|\widehat{\md_s f_{j,m,\flat}}(\xi)|^2\,d\xi\,ds
\lesssim \lambda^{-\delta_j^*}\lambda^{-2\gamma}.
\end{equation}
The factor of $\lambda^{-2\gamma}$ arises because $\norm{f_{j,m}}_{L^2}^4
= O(\lambda^{-2\gamma})$, which is a consequence of the hypothesis
that $\norm{f_j}_{L^\infty} = O(1)$.  The factor $\lambda^{-\delta_j^*}$
encodes an improvement over this trivial bound $O(\lambda^{-2\gamma})$.

\section{Contribution of $(f_{1,\sharp},f_{2,\sharp},f_{3,\sharp},f_{4,\sharp})$}

In this section, we analyze $\scriptt(\bg) = \scriptt(g_1,\dots,g_4)$ with each $g_j$ of the 
local exponential monomial form
\begin{equation} \label{gjsum}
g_j(x) = 
\sum_{m\in\integers} \eta_m(x) h_{j,m}(x)e^{i\alpha_{j,m} x}
\end{equation}
and satisfying
\begin{equation} \label{gjproperties} \left\{ \begin{aligned}
	&\alpha_{j,m}\in\reals,
	\\& |\alpha_{j,m}| = O(\lambda),
\\& |\alpha_{1,m}|\asymp\lambda,
	\\& \norm{h_{j,m}}_{C^N} = O_N(\lambda^{(\gamma+\tau_0)N})\ \text{ for every $N\ge 0$.},
\end{aligned} \right. \end{equation}
with the auxiliary functions $\eta_m$ introduced in \S\ref{subsection:Local_decomp}.
Thus $g_j$ is of the form $f_{j,\sharp}$ in the local $\flat/\sharp$ decomposition,
but with a single summand for each index $m$; in this respect, the title
of this section is a misnomer.
Multilinearity of the form $\scriptt$ will be used in \S\ref{section:final_conclusion} to 
reduce matters to this local monomial case. 

\begin{proposition} \label{prop:sharp}
There exist $\kappa>0$ and $C<\infty$ such that for any $\bg$
of the form \eqref{gjsum} satisfying \eqref{gjproperties},
\begin{equation} |\scriptt(\bg)| \le C\lambda^{-\kappa}.  \end{equation}
\end{proposition}

$\kappa$ depends only on $\Phi$ and on the choices of $\gamma$ and of $\tau_0$.
The constant $C$ also depends on the constants implicit in \eqref{gjproperties}.
Neither $\kappa$ nor $C$ depends on $\lambda$, nor on other properties of $\bg$.

In proving Proposition~\ref{prop:sharp},
we may assume that for each $j$, the sum in \eqref{gjsum}
extends over values of $m\in\integers$ that all lie in a common congruence class
modulo $3$. We reduce to this situation by decomposing the sum over all $m\in\integers$
into $3$ subsums and invoking the multilinearity of $\scriptt$. 
Doing this for each index $j\in\four$
results in a decomposition of $\scriptt(\bg)$ into
$81$ terms. It suffices to examine one such term.
Thus we may assume henceforth that each $g_j$
is a sum \eqref{gjsum} over $m$ lying in one such congruence class.
The terms in such a sum have supports in pairwise disjoint intervals $I_{m}^*$.

\begin{proof}[Proof of Proposition~\ref{prop:sharp}]
The multilinearity of $\scriptt$ gives an expansion
\begin{equation} \label{Tsharpsum}
\scriptt(\bg) = \sum_{\bm\in\integers^4} I_\bm
\end{equation}
with
\begin{equation} \label{Imintegral}
I_\bm = \int e^{i\Phi_\bm(x)}
\prod_{j=1}^4 h_{j,m_j}(\varphi_j(x)) \, \eta_{m_j}(\varphi_j(x))\,dx
\end{equation}
and
\begin{equation} \Phi_\bm(x)  = \sum_{j=1}^4 \alpha_{j,m_j}\varphi_j(x).  \end{equation}
In the sum \eqref{Tsharpsum},
each $m_j\in\integers$ is implicitly
restricted to some congruence class modulo $3$, which depends on $j$.
Terms with different indices $m_j$ thus have disjoint supports.

We say that $\bm = (m_1,\dots,m_4) \in\integers^4$ is interacting
if there exists $x\in B$ such that $\varphi_j(x)\in I_{m_j}^*$
for each $j\in\four$. If $\bm$ is not interacting
then the integrand in \eqref{Imintegral} vanishes identically, so $|I_\bm|=0$.
While the support of each $f_j$ potentially brings $O(\lambda^\gamma)$
indices $m_j$ into play, for a total of $O(\lambda^{4\gamma})$
tuples $\bm$ for which at least one factor $f_{j,m_j}\circ\varphi_j$
could be nonzero, only $O(\lambda^{2\gamma})$
tuples $\bm$ actually interact.
If $\bm$ interacts, then any two of the four indices $m_j$ determine the other two,
up to uniformly bounded additive ambiguity.

Trivially
\begin{equation} \label{Ibmbound} |I_\bm| = O(\lambda^{-2\gamma}) \end{equation}
uniformly in all parameters, since the integrand is $O(1)$
and is supported in a cube of sidelength $O(\lambda^{-\gamma})$ in $\reals^2$.
Together with the bound $O(\lambda^{2\gamma})$ for the number of interacting tuples $\bm$,
this gives $\scriptt(\bg) = O(1)$.
We seek to improve upon this trivial bound by a factor $O(\lambda^{-\tilde\delta})$,
for some $\tilde\delta>0$,
by exploiting the oscillation of $\Phi_\bm$ to improve the upper bound \eqref{Ibmbound} 
for most indices $\bm$.  However, for any particular interacting $\bm$ there exists 
a choice of parameters $\alpha_{j,m_j}$ for which the phase $\Phi_\bm$
is essentially stationary, whence there is no significant improvement.
Thus we must show that no matter how the ensemble of parameters $\alpha_{j,m_j}$
is chosen, within the context of our hypotheses, 
there is sufficient oscillation to create significant cancellation
for all but a relatively small minority of the interacting tuples $\bm$.

Let $\rho_0>0$ be a positive quantity, sufficiently small to satisfy
\begin{equation} \tau_0+\gamma+\rho_0<1.  \end{equation}
Let $\rho>0$ be a small parameter that belongs to $(0,\rho_0]$,
sufficiently small to satisfy various constraints imposed in the course of the proof.	

For each interacting $\bm\in\integers^4$, 
choose $z_\bm\in\reals^2$ satisfying 
\begin{equation}\varphi_j(z_\bm)\in I_{m_j}^*\ \forall\,j\in\four.\end{equation}
By repeatedly integrating by parts, exploiting the
upper bounds \eqref{fsharpproperties} for derivatives of $h_{j,m,n}$,
we conclude that
\begin{equation}
I_\bm = O(\lambda^{-N})\ \forall\,N<\infty
\end{equation}
unless $\Phi_\bm$ is approximately stationary in the sense that
\begin{equation}
|\nabla \Phi_\bm(z_\bm)| \le \lambda^{\gamma+\tau_0+\rho},
\end{equation}
that is, unless
\begin{equation} \label{stationarity4}
\big|\sum_{j=1}^4 \alpha_{j,m_j}\nabla\varphi_j(z_\bm)\big|
\le\lambda^{\gamma + \tau_0+\rho}.
\end{equation}
Indeed, each integration by parts yields a factor $O(\lambda^{-\rho})$,
so that $N'$ integrations by parts gives a bound $O(\lambda^{-N'\rho})$.

Define $\scriptn$ to be the set of all interacting $\bm\in\integers^4$
that satisfy the stationarity condition \eqref{stationarity4}.  Thus 
\begin{equation} \label{scriptNenters}
	|\scriptt(\bg)| \lesssim \lambda^{-N} + \lambda^{-2\gamma}\, \#(\scriptn),
\end{equation}
with $\#(\scriptn)$ denoting the cardinality of $\scriptn$.
Our goal is a bound of the form $\#(\scriptn) = O(\lambda^{2\gamma-\tilde\delta})$.

Define
\begin{equation} \label{Fjdefn}
F_j(x) = \lambda^{-1} \alpha_{j,m}
\ \text{ for each $x\in I_m^*= I_{m-1}\cup I_m\cup I_{m+1}$}
\end{equation}
for each $m\in\integers$ in the appropriate congruence class modulo $3$.
Define $\bF = (F_1,F_2,F_3)$.
The condition $|\alpha_{1,m}|\gtrsim\lambda$
for every $m$ implies that $|F_1(y)|\gtrsim 1$
for every $y\in I_m^*$.  


If the stationarity condition \eqref{stationarity4} holds, and if
$\bm\in\integers^4$ varies only over a single congruence class as indicated above, then
\begin{equation}\label{locallyglobalstationarity}
\big| \sum_{j=1}^4 (F_j\circ\varphi_j)(x)\,\nabla\varphi_j(x) \big|
= O\big(\lambda^{-1+\gamma +\tau_0 +\rho} + \lambda^{-\gamma}\big).
\end{equation}
Indeed, for any $x$ in the support of $\prod_j (\eta_{m_j}\circ\varphi_j)$,
\begin{equation}
\big| \alpha_{j,m_j}\nabla\varphi_j(x)
- \alpha_{j,m_j}\nabla\varphi_j(z_\bm) \big|
= O(\lambda\cdot \lambda^{-\gamma})
\end{equation}
since $\nabla^2 \varphi_j= O(1)$
and $|\alpha_{j,m_j}| = O(\lambda)$.
Invoking \eqref{stationarity4} and taking into account the factor 
of $\lambda^{-1}$ in the definition of $F_j$ gives \eqref{locallyglobalstationarity}.

Since $\gamma> \tfrac12$ with strict inequality,
$\lambda^{-\gamma} \ll \lambda^{-1+\gamma}\ll \lambda^{-1+\gamma+\tau_0+\rho}$
and therefore \eqref{locallyglobalstationarity} implies that
\begin{equation} \label{willdefinesublevelset}
\big| \sum_{j=1}^4 (F_j\circ\varphi_j)(x)\,\nabla\varphi_j(x) \big|
= O(\lambda^{-1+\gamma + \tau_0 +\rho}).
\end{equation}

Defining the sublevel set
\begin{equation}
S(\bF,\eps)
= \big\{x\in\reals^2: 
\big| \sum_{j=1}^4 (F_j\circ\varphi_j)(x)\,\nabla\varphi_j(x) \big|
<\eps\big\},
\end{equation}
we have arranged that
\begin{equation} \label{SFenters}
\lambda^{-2\gamma}\cdot\#(\scriptn)
\lesssim |S(\bF,\,C\lambda^{-(1-\gamma-\tau_0-\rho)})|
\end{equation}
for a certain constant $C<\infty$, 
where $\#(\scriptn)$ denotes the cardinality of $\scriptn$.
Thus an upper bound for
$\lambda^{-2\gamma}\cdot\#(\scriptn)$
will be a consequence of an upper bound for 
$|S(\bF,\,C\lambda^{-(1-\gamma-\tau_0-\rho)})|$.
Since $\gamma+\tau_0+\rho<1$ and $|F_1(\varphi_1(x))|\gtrsim 1$,
the inequality \eqref{willdefinesublevelset}
can hold only if there is significant cancellation in the sum
$\sum_{j\in\three} (F_j\circ\varphi_j)(x)$.

By making a smooth change of variables, we may
assume without loss of generality that 
\begin{equation} \varphi_4(x_1,x_2)\equiv x_1.\end{equation}
Then  writing $a_j = \partial\varphi_j/\partial x_2$,
\begin{equation}
S(\bF,\eps)\subset
S'(F_1,F_2,F_3,\eps) = 
\big\{x: 
\big| \sum_{j=1}^3 a_j(x) (F_j\circ\varphi_j)(x) \big|
<\eps\big\}
\end{equation}
for any $\eps>0$.
Observe that the additive inequality defining $S(\bF,\eps)$
involves only three functions $F_j$, and has variable coefficients $a_j$.

Let $\Omega\subset\reals^2$ be a closed bounded open ball or parallelepiped 
with nonempty interior.
For each $j\in\three$ let
$a_j:\overline{\Omega}\to\reals$ be a $C^\omega$ function
in a neighborhood of $\Omega$.
Likewise, let $\varphi_j:\Omega\to\reals^1$ be $C^\omega$ submersions. 
To any three-tuple $\bff = (f_j: j\in\three)$
of Lebesgue measurable functions $f_j:\Omega\to\reals$,
and to any $\eps>0$, associate the sublevel set
\begin{equation}
S'(\bff,\eps) = \{x\in\Omega: 
\big|\sum_{j=1}^3 a_j(x) (f_j\circ\varphi_j)(x)\big| <\eps\}.
\end{equation}

We are in a position to invoke the following result of the companion paper \cite{sublevel4}.
For each $j\in\four$, let $V_j$ be a real analytic vector field in $\tilde B$
that vanishes nowhere and satisfies $V_j(\varphi_j)\equiv 0$.


 
\begin{theorem} \label{thm:sublevel}
Let $a_j,\varphi_j\in C^\omega(\tilde B)$.
Assume that the coefficients $a_j$ vanish nowhere in $B$,
and that for each $i\ne j\in\{1,2,3\}$,
$\nabla\varphi_i(x)$ and $\nabla\varphi_j(x)$ are linearly independent at each $x\in B$.

Assume that for any nonempty open set $U\subset \tilde B$,
if $f_j:\varphi_j(U)\to\reals$ are real analytic
and $\sum_{j=1}^3 a_j\cdot (f_j\circ\varphi_j)\equiv 0$ in $U$
then each $f_j\equiv 0$ in $\varphi_j(U)$.

Assume that for each permutation $(i,j,k)$ of $(1,2,3)$,
for any $\tau\in\reals$ and any nonempty open set $U\subset\tilde B$,
the function $\frac{a_i |V_k\varphi_i|^\tau}{a_j |V_k\varphi_j|^\tau}$
cannot be expressed in $U$ as $\frac{h_i\circ\varphi_i}{h_j\circ\varphi_j}$
for any real analytic functions $h_i,h_j$.

There exist $C<\infty$ and $\tau>0$ with the following property.
Let $\eps>0$ be arbitrary.
For any ordered triple $\bff$ of Lebesgue measurable functions
satisfying $|f_3(y)|\ge 1$ for every $y\in\varphi_3(B)$,
\begin{equation} \label{ineq:mainbound}
|S(\bff,\eps)|\le C\eps^\tau.
\end{equation}
\end{theorem}

We wish to apply
Theorem~\ref{thm:sublevel} with $f_j=F_j$ and $a_j = \partial\varphi_j/\partial x_2$,
so must verify that its hypotheses are satisfied.
We have arranged that $|F_1|\gtrsim 1$ at every point of its domain.
The transversality hypothesis that $\nabla\varphi_j$ and $\nabla\varphi_4$ are
everywhere linearly independent for each $j\in \three$
implies that $a_j$ vanishes nowhere. 

We must verify that
the main hypothesis of Theorem~\ref{thm:sublevel} is satisfied.
It states that any $C^\omega$ local solution $\bg = (g_1,g_2,g_3)$  of the equation
\begin{equation}
\sum_j (g_j\circ\varphi_j)(x)\,\cdot\,\frac{\partial\varphi_j}{\partial x_2}(x)\equiv 0
\end{equation}
in any open set must vanish identically.
So let $\bg\in C^\omega$ satisfy this equation in a nonempty small open set. 
Choose an antiderivative $G_j$ for each $g_j$.  Then
\begin{equation}
\frac{\partial}{\partial x_2}
\big(\sum_{j=1}^3 G_j\circ\varphi_j\big)\equiv 0.
\end{equation}
Since $\varphi_4(x_1,x_2)\equiv x_1$,
this means that $\sum_{j=1}^3 (G_j\circ\varphi_j)$
can be expressed as $-G_4\circ\varphi_4$ for a certain function $G_4$.
Then $\sum_{j=1}^4 (G_j\circ\varphi_j)$ vanishes identically in an open set.
By the main hypothesis of Theorem~\ref{maintheorem}, this implies that each function $G_j$
is locally constant. Therefore the derivatives $g_j$ vanish identically.
Finally, the hypothesis concerning
$\frac{a_i |V_k\varphi_i|^\tau}{a_j |V_k\varphi_j|^\tau}$,
with $a_i = \partial\varphi_i/\partial x_2 = V_4(\varphi_i)$
is Hypothesis~\ref{auxhyp3} of Theorem~\ref{maintheorem}. 

Thus all hypotheses of Theorem~\ref{thm:sublevel} are indeed satisfied. 
Since the parameters were chosen to satisfy $\tau_0+\gamma+\rho\le \tau_0+\gamma+\rho_0<1$,
we conclude that
\begin{equation}
|S'(\bF,\lambda^{-(1-\gamma-\tau_0-\rho)})|
= O(\lambda^{-\tilde\delta})
\end{equation}
with $\tilde\delta = (1-\gamma-\tau_0-\rho_0)\varrho>0$
for a certain exponent $\varrho>0$.
Combining this with \eqref{scriptNenters} and \eqref{SFenters}, 
we conclude that
\begin{equation} \label{propsharp:conclusion}
|\scriptt(\bg)| \lesssim \lambda^{-\tilde\delta} + C_N \lambda^{-N} 
\end{equation}
for every $N<\infty$. 
This completes the proof of Proposition~\ref{prop:sharp}.
\end{proof}



\section{Contributions of functions $f_{j,\flat}$} \label{section:flat}


\begin{proposition} \label{prop:flat}
For each $\delta>0$ there exist $\tau,C\in(0,\infty)$ with the following property.
Suppose that $\norm{f_j}_{L^\infty}\le 1$ for each $j\in\four$.
	Suppose that $f_1 = \sum_m \eta_m\cdot f_{1,m}$ where $f_{1,m}$ satisfies
\begin{equation} \label{flathypothesis}
\int_{\reals} \int_{|\xi| \le \lambda^{\tau_0+\gamma}} 
|\widehat{\md_s f_{1,m}}(\xi)|^2\,d\xi\,ds
\lesssim \lambda^{-\delta}\lambda^{-2\gamma}
\end{equation}
for each $m\in\integers$.  Then
\begin{equation} |\scriptt(\bff)| \le C\lambda^{-\tau}.  \end{equation}
\end{proposition}

To begin the proof of Proposition~\ref{prop:flat}, expand
\begin{equation}
\scriptt(\bff) = \sum_{\bm} \scriptt(\eta_{m_j}f_j: j\in\four)
\end{equation}
with the sum extending over all interacting
$\bm = (m_1,m_2,m_3,m_4)$. 
Change variables to $(x_1,x_2)\in\reals^2$
so that $\varphi_4(x_1,x_2)\equiv x_1$,
if necessary introducing a finite partition of unity independent of $\lambda$ 
so that such a transformation is possible on the support of each summand.

By the Cauchy-Schwarz inequality,
\begin{equation} \label{afterCS1}
|\scriptt(\bff)|
\lesssim 
\lambda^{-\gamma/2} \sum_\bm 
\Big(
\iiint \prod_{j\in\three} f_j(\varphi_j(x_1,x'_2)) \overline{f_j}(\varphi_j(x_1,x_2))
\,\,\zeta_\bm(x_1,x_2,x'_2)\,dx_1\,dx_2\,dx'_2
\Big)^{1/2}
\end{equation}
with 
\begin{equation} \zeta_\bm(y,t,t') = \prod_{j=1}^4 \, \eta_{m_j}(y,t)\eta_{m_j}(y,t'), \end{equation}
assuming as we may that all $\eta_m$ are real-valued.

\begin{notation}
\begin{equation}
f_j^s(x) = \scriptd_s(f_j)(x) =  f_j(x+s)\overline{f_j}(x).
\end{equation}
\end{notation}

\begin{notation}
$Q_\bm\subset\reals^2$
denotes a cube with sides parallel to the coordinate axes
of sidelength $\lambda^{-\gamma}$
such that for any $x'_2\in\reals$, 
the function $(x_1,x_2)\mapsto \zeta_\bm(x_1,x_2,x'_2)$
is supported in the cube $Q_\bm^*$
of sidelength $C\lambda^{-\gamma}$ concentric with $Q_\bm$, 
where $C<\infty$ depends only on $\Pphi$.
$\barz_\bm$ denotes some fixed element of $Q_\bm$.
\end{notation}

Substitute $x'_2 = x_2+s$ in the triple integral in \eqref{afterCS1},
changing variables from $(x_1,x_2,x'_2)$ to $(x_1,x_2,s)$.
Recalling that $\varphi_4(x_1,x_2)\equiv x_1$, write
\[ D\varphi_j(x) = \frac{\partial \varphi_j}{\partial x_2}(x)\ne 0.\]
For $x\in Q_\bm^*$,
\begin{equation} \label{linearapprox}
f_j(\varphi_j(x_1,x_2)) \overline{f_j}(\varphi_j(x_1,x'_2))
= f_j^{sD\varphi_j(\barz)}(\varphi_j(x))
+ O(\lambda^{1-2\gamma})
\end{equation}
for each $j\in\three$.
Indeed, as in \cite{triosc} and \cite{CDR}, 
\eqref{linearapprox} is a consequence of the assumption that $\widehat{f_j}(\xi)$
is supported where $|\xi| = O(\lambda)$;
this restriction on the support of the Fourier transform implies that
$\norm{\nabla f_j}_{L^\infty} = O(\lambda \norm{f_j}_{L^\infty}) = O(\lambda)$,
and replacing $\varphi_j$ by its first order Taylor polynomial
results in a perturbation of the argument $\varphi_j$ of $f_j$ 
of size $O(\lambda^{-2\gamma})$. 

Therefore
$|\scriptt(\bff)|$ is majorized by $C \lambda^{(1-2\gamma)/2}$ plus
\begin{equation} \label{ineq:plus}
C\lambda^{-\gamma/2}   \sum_\bm \Big(
\int_{|s| = O(\lambda^{-\gamma})}
\Big|
\int_{\reals^2} \prod_{j\in\three} f_j^{sD\varphi_j(\barz)}(\varphi_j(x)) 
\,\,\zeta_\bm(x,s)\,dx
\Big| \,ds \Big) ^{1/2}
\end{equation}
with $\zeta_\bm$ renamed but retaining its essential properties. 
Since $\gamma>\tfrac12$, the exponent $(1-2\gamma)/2$ is negative.

By Cauchy-Schwarz, the preceding line is in turn majorized by
\begin{equation} \label{tbm1}
C\lambda^{\gamma/2}   \Big( \sum_\bm 
\int_{|s| = O(\lambda^{-\gamma})}
\Big|
\int_{\reals^2} \prod_{j\in\three} f_j^{sD\varphi_j(\barz)}(\varphi_j(x)) 
\,\,\zeta_\bm(x,s)\,dx
\Big| \,ds \Big)^{1/2}.
\end{equation}
The sum extends only over interacting indices $\bm$,
and there are only $O(\lambda^{2\gamma})$ interacting indices,
so this application of Cauchy-Schwarz
incurs a factor $O(\lambda^\gamma)$ which combines
with the factor $O(\lambda^{-\gamma/2})$ in \eqref{ineq:plus}
to give the initial factor of $\lambda^{\gamma/2}$.

Associate the factors $\eta_{m_j}$ that were incorporated into the factors
$\zeta_\bm$ with the functions $f_j$, and
expand in local Fourier series
\begin{equation} \label{localFourier}
\eta_{m}^2(x) f_{j,m}^s(x) 
= \eta_{m}(x)
\sum_{k_j} a_{j,m,s,k_j} e^{i\pi\lambda^\gamma k_j x}.
\end{equation}
The coefficients satisfy
\begin{equation} \label{bessel}
\sum_{k_j\in\integers} |a_{j,m,s,k_j}|^2 = O(1)
\end{equation}
uniformly in all parameters $j,m,s$.

The property \eqref{flathypothesis} of $f_1$ now comes into play. By \eqref{keyflatproperty},
with $\lambda^\gamma k_1$ here playing the role there of the Fourier variable $\xi$, 
\begin{equation}
\int_{|s| = O(\lambda^{-\gamma})}
\sum_{ |k_1| \le \lambda^{\tau_0}}
|a_{1,m_1,s,k_1}|^2\,ds
= O(\lambda^{-\gamma-\delta}).
\end{equation}
Decompose 
\begin{equation}
\eta^2_{m_1}(x) f_{1,m_1}^s(x) 
= \eta_{m_1}(x) f_{1,m_1}^{s,\star}(x)
+ \eta_{m_1}(x) f_{1,m_1}^{s,\dagger}(x)
\end{equation}
with
\begin{equation}
f_{1,m_1}^{s,\star}(x)
= \sum_{|k_1|>\lambda^{\tau_0}} a_{1,m_1,s,k_1} e^{i\pi\lambda^\gamma k_1 x}
\end{equation}
and with 
$f_{1,m_1}^{s,\dagger}(x)$
equal to the sum over those $k_1$ satisfying $|k_1|\le\lambda^{\tau_0}$.

The second summand satisfies
\begin{equation}
\int_{|s| = O(\lambda^{-\gamma})}
\norm{\eta_{m_1}f_{1,m_1}^{s,\dagger}}_{\lt}^2\,ds
= O(\lambda^{-2\gamma-\delta}).
\end{equation}
It follows that
for each $\bm$, the contribution of $\eta_{m_1}f_{1,m_1}^{s,\dagger}$ to 
\[\int_{|s| = O(\lambda^{-\gamma})}
\Big| \int_{\reals^2} \prod_{j\in\three} f_j^{sD\varphi_j(\barz)}(\varphi_j(x)) 
\,\,\zeta_\bm(x,s)\,dx \Big| \,ds\]
is $O(\lambda^{-3\gamma-\delta/2})$. 
Therefore if $f_1^{sD\varphi_j}$ is replaced by $f_1^{sD\varphi_j,\dagger}$
in \eqref{ineq:plus}, then upon summation over all interacting $\bm\in\integers^4$,
the resulting contribution to the upper bound
for $|\scriptt(\bff)|$ is $O(\lambda^{-\delta/4})$.  
Therefore we may henceforth 
assume that 
all indices $k_1$ appearing in the local Fourier expansion \eqref{localFourier} of $f_1^s$
satisfy $|k_1|>\lambda^{\tau_0}$.

Consider any interacting $\bm$.
Inserting the three local Fourier expansions into the inner integral in \eqref{tbm1},
that is, into
\begin{equation}  \label{inner}
\int_{\reals^2} \prod_{j\in\three} f_j^{sD\varphi_j(\barz)}(\varphi_j(x)) 
\,\,\zeta_\bm(x,s)\,dx,
\end{equation}
produces the sum of oscillatory integrals
\begin{equation} \label{anoscintegral}
\sum_\bk \prod_{j\in\three} a_{j,m_j,sD\varphi_j(\barz),k_j}
\int_{\reals^2} e^{i\pi\lambda^\gamma\Psi_\bk(x)}
\,\,\zeta_\bm\,dx,
\end{equation}
with phase functions
\begin{equation}
\Psi_\bk(x) = \sum_{j=1}^3 k_j\,\varphi_j(x).
\end{equation}
The sum extends over those $\bk=(k_1,k_2,k_3) \in\integers^3$ satisfying $|k_1|> \lambda^{\tau_0}$.

Let $\delta_1,\rho>0$ be small quantities to be chosen below.
$\delta_1$ will depend on $\delta$, and $\rho$ will depend on $\delta_1$.
The integral \eqref{anoscintegral} is $O(\lambda^{-N})$
for every $N<\infty$, as follows by repeated integrations by parts,
unless $\Psi_\bk$ satisfies the stationarity condition 
\begin{equation} \label{stationarity}
|\nabla\Psi_\bk(\barz_\bm)| = O(\lambda^{\rho}),
\end{equation}
that is, unless
\begin{equation} \label{stationarity2}
\big| \sum_{j=1}^3 k_{j} \cdot \nabla\varphi_j(\barz_\bm) \big|
= O(\lambda^\rho ).
\end{equation}

When \eqref{stationarity2} holds,
each $k_{i}$ determines the other two quantities
$k_{j}$ up to additive ambiguity $O(\lambda^\rho)$.
That is, if \eqref{stationarity2} holds for $\bk=(k_1,k_2,k_3)$
and for $(k'_1,k'_2,k'_3)$,
and if $k_i=k'_i + O(\lambda^\rho)$, then $|k_j-k'_j| = O(\lambda^\rho)$
for each $j\in\three$.
This is a consequence of the transversality hypothesis,
which guarantees that any $2\times 2$ submatrix of the
$3\times 2$ matrix 
\[\begin{pmatrix} \nabla\varphi_1(x) \\ \nabla\varphi_2(x) \\ \nabla\varphi_3(x) \end{pmatrix}\]
is nonsingular for every $x\in B$.

The parameter $\delta_1>0$ now comes into play.
Decompose 
\begin{equation}
f_{j,m_j}^s = g_{j,m_j}^s+h_{j,m_j}^s
\end{equation}
where $h_{j,m_j}^s$ is the sum of those terms 
in \eqref{localFourier}
satisfying
\begin{equation}  \label{smallFouriercoeffs}
|a_{j,m_j,s,k_j}| \le \lambda^{-\delta_1}
\end{equation}
and if $j=1$, also satisfying $|k_1|>\lambda^{\tau_0}$.

As in \cite{triosc},
it is a consequence of the transversality hypothesis,
\eqref{stationarity}, and the local Fourier coefficient bound \eqref{bessel} that
\eqref{inner} equals
\begin{equation}  \label{innerg1}
O(\lambda^{-2\gamma-\delta_1 + C\rho}) + 
\int_{\reals^2} \prod_{j\in\three} g_j^{sD\varphi_j(\barz)}(\varphi_j(x)) 
\,\,\zeta_\bm(x)\,dx.
\end{equation}
For each $(j,m_j)$, the number of frequencies $k_{j,m_j}$
for which \eqref{smallFouriercoeffs} fails to hold,
is $O(\lambda^{2\delta_1})$ by Bessel's inequality \eqref{bessel}.
Consequently, for each index $j$, 
the function $\sum_{m_j} g_{j,m_j}^s$ is a sum of $O(\lambda^{2\delta_1})$
functions of the form
\begin{equation}
\tilde g_j(x,s) = 
\sum_{m\in\integers}
\eta_m(x)\, a_{j,m,s}\, e^{i\pi\lambda^\gamma k_j(m,s)x},
\end{equation}
with complex scalar coefficients $a_{j,m,s}=O(1)$,
and with real frequency functions that satisfy
the supplementary condition $|k_1(m,s)|>\lambda^{\tau_0}$ for every $m$ for $j=1$.
We refer to such a function $\tilde g_j^s$ as being of local exponential
monomial form.
Altogether, we have a sum of the contributions of $O(\lambda^{6\delta_1})$ such functions.
As above, we ensure that the terms in this sum are disjointly
supported, by partitioning those $m\in\integers$
according to their congruence classes modulo $3$ for each index $j$.

The functions $(m,s)\mapsto k_j(m,s)$ now take on the leading role
in the analysis.
If we can show that there exists $\tilde\delta>0$, independent of $\rho$, such that
\begin{equation} \label{desideratum1}
\sum_{\bm} \int_{|s| = O(\lambda^{-\gamma})}
\Big| \int_{\reals^2} \prod_{j\in\three} \tilde g_j(\varphi_j(x),sD\varphi_j(\barz_\bm))
\,\zeta_\bm(x,s)\,dx \Big| \,ds
= O(\lambda^{-\gamma-\tilde\delta})
\end{equation}
uniformly for any three functions $\tilde g_j$
of monomial form with $|k_1(m,s)|\ge\lambda^{\tau_0}$ for all $(m,s)$,
then we may conclude that
\begin{equation} \label{desideratum2}
\sum_{\bm} \int_{|s| = O(\lambda^{-\gamma})}
\Big| \int_{\reals^2} \prod_{j\in\three} g_j^{sD\varphi_j(\barz_\bm)}
\,\zeta_\bm(x,s)\,dx \Big| \,ds
= O(\lambda^{-\gamma-\tilde\delta+6\delta_1}).
\end{equation}
The exponent $\delta_1$ is at our disposal, but is not permitted
to depend on $\rho$. Since $\tilde\delta$ is independent of $\rho$,
we may choose $\delta_1<\tfrac16\tilde\delta$
to obtain a favorable upper bound for this term.
Henceforth we write $g_{j}^s$ rather than $\tilde g_{j,m_j}(\cdot,s)$
to indicate functions with these properties.

\begin{definition} \label{bigdefn}
Let $(g_j^s: j\in\three)$ be an ordered triple of functions of local exponential
monomial form satisfying $|k_1(m,s)|\ge \lambda^{\tau_0}$.
Let $f_1$ satisfy $\norm{f}_{L^\infty} = O(1)$. 
The associated set $\scriptm = \scriptm(g_1,g_2,g_3,f_1)$ is the set of all $(\bm,s)$
such that $\bm\in\integers^4$ is interacting, $s\in\reals$
satisfies $|s| = O(\lambda^{-\gamma})$,
$(\bm,s)$ satisfies the stationarity condition 
\begin{equation} \label{stationarity_restated}
\big|\sum_{j=1}^3 k_j(\bm,s)\nabla\varphi_j(\barz_\bm)\big| = O(\lambda^\rho),
\end{equation}
and
\begin{equation} \label{extrainfo}
\Big| \int \eta_{m_1}(y) f_1(y+s)\overline{f_1(y)}
\,e^{-i\pi\lambda^\gamma k_1(m_1,s)\,y}\,dy\Big|
\gtrsim \lambda^{-\gamma-\delta_1}.
\end{equation}
\end{definition}

Denote by $|\scriptm|$ the measure of $\scriptm$,
with respect to counting measure on $\bm$,
as $\bm$ varies over all interacting elements of $\integers^4$,
and with respect to Lebesgue measure on $s\in\reals$.
Since there are $O(\lambda^{2\gamma})$ interacting indices $\bm$,
The trivial bound for the measure of $\scriptm$ is
$|\scriptm| = O(\lambda^{2\gamma}\lambda^{-\gamma}) = O(\lambda^\gamma)$,
with a factor $O(\lambda^{2\gamma})$
reflecting the number of interacting indices $\bm$,
and a factor $O(\lambda^{-\gamma})$
reflecting the restriction $|s| = O(\lambda^{-\gamma})$.

Denote by $|\scriptm|^\star$ the supremum of $|\scriptm|$,
taken over the set of all possible
tuples $(\tilde g_j: j\in\three)$ of local exponential monomials
enjoying the properties indicated above and all $f_1$.
In these terms, we have shown that
\begin{multline} 
\sum_{\bm} \int_{|s| = O(\lambda^{-\gamma})}
\Big| \int_{\reals^2} \prod_{j\in\three} \tilde g_j(\varphi_j(x),sD\varphi_j(\barz_\bm))
\,\zeta_\bm(x,s)\,dx \Big| \,ds
\\
= O(\lambda^{-N}) + O(\lambda^{-2\gamma}|\scriptm|^\star).
\end{multline}
We note that the coefficient implicit in the notation $O(\lambda^{-N})$
does depend on $\rho$, but this dependence is harmless since the exponent $N$ is 
independent of $\rho$, indeed, is arbitrarily large.
Therefore the quantity \eqref{tbm1} 
associated to the ordered triple $(f_j^s: j\in\three)$ with which
we began \S\ref{section:flat} is majorized by
\begin{equation} \label{many_term_bounds}
O\big(\lambda^{(1-2\gamma)/2}
+ \lambda^{-\delta/4}
+ \lambda^{-\delta_1/2}\lambda^{C\rho} 
	+\lambda^{-N}\big)
+ O\big(\lambda^{\gamma/2}\lambda^{6\delta_1}
(\lambda^{-2\gamma} |\scriptm|^\star)^{1/2}\big)
\end{equation}
for any $N<\infty$.
The first term is favorable since $\gamma>\tfrac12$,
and the second since $\delta>0$. The third term is favorable
provided that $\rho$ is chosen to be sufficiently small relative to $\delta_1$.

To complete the proof of Proposition~\ref{prop:flat}, it suffices to
establish the following upper bound for $|\scriptm|^\star$.

\begin{lemma} \label{lemma:flatsublevel}
There exists $\delta_0>0$ such that whenever $0<\delta_1\le\delta_0$,
and whenever $\rho>0$ is sufficiently small,
there exists $\varrho>0$ such that for any 
$\scriptm,\bg,f_1,\delta_1,\rho$ as in Definition~\ref{bigdefn},
\begin{equation} \label{ineq:mainsublevel} |\scriptm| =O(\lambda^{\gamma-\varrho}).  \end{equation}
\end{lemma}

If $\delta_1$ is chosen to be $\le\min(\delta_0,\tfrac1{24}\varrho)$
then the last term in \eqref{many_term_bounds}
is $O(\lambda^{-\varrho/4})$, and thus Lemma~\ref{lemma:flatsublevel}
yields the required upper bound for \eqref{tbm1}.

Lemma~\ref{lemma:flatsublevel} will be applied with the parameter $\rho$ in the discussion above
satisfying $\rho\le\rho_0$. The upper bound for $|\scripts|$
will provide an upper bound for $|\scriptm|^*$ by the stationarity condition \eqref{stationarity},
since $\lambda^\rho\le\lambda^{\rho_0}$ for large $\lambda$. The formulation emphasizes that $\varrho$
is then independent of $\rho$, allowing us to choose $\rho$ to be arbitrarily small
and hence sufficiently small relative to $\delta,\delta_1$.

The trivial bound in Lemma~\ref{lemma:flatsublevel} is $|\scripts| = O(\lambda^{\gamma})$.
As in the proof of Proposition~\ref{prop:sharp}, we have reduced matters
to an upper bound for the measure of a sublevel set;
Lemma~\ref{lemma:flatsublevel} plays the same role in the proof of Proposition~\ref{prop:flat} 
that Theorem~\ref{thm:sublevel} plays in the proof of Proposition~\ref{prop:sharp}.

\medskip \noindent {\bf Remark.}\ 
The assumption \eqref{extrainfo} in Definition~\ref{bigdefn}
would be unnecessary for Lemma~\ref{lemma:flatsublevel}, 
and \S\ref{section:mostsigma} and \S\ref{section:conclusionofproof} below
would likewise be superfluous, if auxiliary hypothesis~\ref{auxhyp1}
were strengthened to exclude nonzero real analytic 
solutions of $\sum_{j\ne k} (F_j\circ\varphi_j)\,(V_k\varphi)^\sigma \,\nabla\varphi_j=0$
for all $\sigma\in\reals$, rather than merely for $\sigma=1$.

Let $\scriptm,\bg,f_1,\delta_1,\rho$ be as in Lemma~\ref{lemma:flatsublevel}.
For each interacting $\bm$, define $\scripts^*(\bm)\subset\reals$
to be the set of $s\in\reals$ that satisfy
$|s| = O(\lambda^{-\gamma})$, \eqref{extrainfo}, and
\begin{equation}
\big| \sum_{j=1}^3 k_j(m_j,s)\,\nabla_x\varphi_j(x)| = O(\lambda^\rho)
\ \forall\,x\in Q_\bm^*.
\end{equation}
Since $\rho<\gamma$, 
the last inequality is equivalent to the stationarity condition
\eqref{stationarity_restated}
in the definition of $\scriptm$,
after appropriate adjustment of the constant factors implicit in
the notations $O(\lambda^\rho)$.
Thus the conclusion \eqref{ineq:mainsublevel} of Lemma~\ref{lemma:flatsublevel} can be restated as
\begin{equation} \label{ineq:mainsublevelrestated}
\sum_{\bm} |\scripts^*(\bm)| = O(\lambda^{\gamma-\varrho}),
\end{equation}
where the summation extends over all interacting indices $\bm\in\integers^4$.
To complete the proof of Proposition~\ref{prop:flat}, it suffices to prove 
\eqref{ineq:mainsublevelrestated}. Our next four sections are devoted to this task.

\section{A functional (approximate) equation} \label{section:functionalequation}

To $\varphi_j$ are associated mappings $\Phi_j:B\times\reals\to\reals^2$
defined by
\begin{equation}
\Phi_j(x,s) = (\varphi_j(x),sD\varphi_j(x))
\end{equation}
where $D = \partial/\partial x_2$.
We next study solutions of the functional equation
\begin{equation} \label{Phisystem}
\sum_{j=1}^3 (f_j\circ\Phi_j)(x,s)\, \nabla \varphi_j(x)=0,
\end{equation}
in which the unknown is an ordered triple
$\bff = (f_j: j\in\three)$
of functions defined on some open subset of $\reals^2\times\reals^1$.
A solution $\bff$ is said to be trivial if it vanishes  identically.

We also study associated sublevel set inequalities,
which quantify the nonexistence of nontrivial solutions.
The results that we develop will be applied to functions $f_j$
defined by $f_j(y,s) = k_j(m,s)$ when $y\in I_m^*$,
and $k_j$ are the functions in the hypothesis of Lemma~\ref{lemma:flatsublevel}.
Throughout this section, all hypotheses of Theorem~\ref{maintheorem}
are implicitly assumed, although certain lemmas rely on only some of those hypotheses.

Here $(x,s) = ((x_1,x_2),s)$ belongs to some open subset of $\reals^2\times\reals^1$,
$\varphi_j$ are $C^\omega$ submersions 
mapping an open subset of $\reals^2$ to $\reals^1$,
and $f_j$ is defined on an open subset of $\reals^2$.
$\nabla=\nabla_x$ denotes the gradient with respect to $x$ alone,
and $D\varphi_j = \partial\varphi_j/\partial x_2$.
Thus each term 
$f_{j}(\varphi_j(x),sD\varphi_j(x)) \cdot \nabla\varphi_j(x)$
is $\reals^2$--valued.

For each $j\in\three$, $D\varphi_j$ vanishes nowhere.
This is a consequence of the transversality hypothesis \eqref{transversehyp}
and the choice of coordinates, in which $\varphi_4(x_1,x_2)\equiv x_1$,
and consequently $D = \partial/\partial x_2$ annihilates $\varphi_4$.


\subsection{Exact solutions}
The equation \eqref{Phisystem} for an ordered
triple $\bff$ of functions $f_j$ is homogeneous with respect to $s$;
if $\bff(x,s)$ is a solution then so is $\bff(x,rs)$
for any $r\in\reals\setminus\{0\}$.


The differential of each $\Phi_j$ with respect to $(x,s)\in\reals^3$
has rank $2$ at each point of its domain, since $D\varphi_j\ne 0$.
Thus each of the three relations
\begin{equation} \label{definefoliations}
\Phi_j(x,s) = \text{ constant $\in\reals^2$}
\end{equation}
defines a foliation in $\reals^3$ with one-dimensional leaves 
indexed by the constants on the right-hand side of \eqref{definefoliations}.
We refer to these leaves in $\reals^3$ as level curves of $\Phi_j$.
For each index $j$, the vector field
\begin{equation} \label{Vjdaggerdefn}
V_j^\dagger = \Big( -\partial_2\varphi_j,\, \partial_1\varphi_j,\, 
s(D\varphi_j)^{-1}\big(\partial_2\varphi_j\cdot\partial^2_{1,2}\varphi_j
- \partial_1\varphi_j\cdot\partial^2_{2,2}\varphi_j \big) \Big)
\end{equation}
vanishes nowhere and annihilates $\Phi_j$; $V_j^\dagger$
is tangent to the leaves of the associated foliation.
Since $\nabla\varphi_j$ are pairwise transverse in $\reals^2$,
it follows from an examination of the vector fields in $\reals^2$
defined by their first two components
that $V_i^\dagger,V_j^\dagger$ are everywhere linearly independent whenever $i\ne j$.
Therefore the transversality hypothesis \eqref{transversehyp} 
ensures that these foliations are everywhere pairwise transverse.
However, the hyperplane defined by $s=0$ is degenerate 
in the sense that all three tangent vectors $V_j^\dagger$ are tangent to 
this hyperplane at each of its points.

Write 
\begin{equation} \reals^3_{\ne 0} = \{(x,s)\in\reals^2\times\reals: s\ne 0\}.
\end{equation}

\begin{lemma} \label{lemma:Comega}
Let $\Phi$ satisfy the transversality hypothesis \ref{transversehyp}.
Let $\bff$ be any ordered triple of Lebesgue measurable functions
that satisfies the functional equation \eqref{Phisystem}
almost everywhere in some open subset $\Omega$ of $\reals^3_{\ne 0}$.
Then each function $f_j\circ\Phi_j$
agrees almost everywhere with a $C^\omega$ function in $\Omega$.
\end{lemma}

\begin{proof}
We work in a small neighborhood of an arbitrary point of $\Omega$,
and all assertions in this proof are concerned with appropriately
defined but unspecified such neighborhoods.
Thus ``$x\in\reals^2$'' means that $x\in\reals^2$ belongs to such
a neighborhood.
We initially simplify matters by assuming that each $f_j$ is continuous.

Any two gradients $\nabla\varphi_j(x)$
are transverse at every $x\in\reals^2$, and consequently the system 
\eqref{Phisystem} of two scalar linear equations in three unknown quantities $f_k$
can be solved locally to express 
\begin{equation} \label{fifj}
(f_1\circ\Phi_1)(x,s) = h_{j}(x)\,\cdot\, (f_j\circ\Phi_j)(x,s)
\end{equation}
with $h_{j}\in C^\omega$ independent of $s$, both for $j=2$ and for $j=3$.
The functions $h_j$ are entirely specified by the datum $(\varphi_j: j\in \four)$.	

If $\Gamma$ is any leaf of the foliation in $\reals^3_{\ne 0}$
defined by level curves of $\Phi_2$, then its image 
$\Phi_1(\Gamma)$ is a $C^\omega$ curve in the two-dimensional domain of $f_1$.
This is a direct consequence of
the transversality of the two vectors $\nabla\varphi_1,\nabla\varphi_2$.
The family of all such $\Phi_1(\Gamma)$ is a two parameter family of
curves in this domain, rather than a foliation.

In the same way, there is a corresponding family of curves $\Phi_1(\Gamma')$,
with $\Gamma'\subset\reals_{\ne 0}^3$ denoting level curves of $\Phi_3$. 
The auxiliary hypothesis~\ref{auxhyp2} of Theorem~\ref{maintheorem}
now enters the discussion for the first time. It
ensures that if $\Gamma,\Gamma'$ are level curves of $\Phi_2,\Phi_3$
respectively that meet at a point $(x,s)\in\reals^3_{\ne 0}$ then their images 
$\Phi_1(\Gamma)$ and $\Phi_1(\Gamma')$ intersect transversely
in $\reals^2$ at $\Phi_1(x,s)$.

Consider any point $(\barx,\bars)\in\reals^3_{\ne 0}=\reals^2\times\reals^1_{\ne 0}$.
Set $a = f_2(\Phi_2(\barx,\bars))$.
Denote by $\bar\Gamma$ the level curve of $\Phi_2$ that passes through $(\barx,\bars)$.
The relation \eqref{fifj} expresses 
the restriction of $f_1$ to $\Phi_1(\bar\Gamma)$
as the constant $a$ multiplied by a real analytic
function, which depends only on the datum $(\varphi_k: k\in\four)$.

Consider those leaves $\Gamma'$ that intersect $\bar\Gamma$.
The restriction of $f_1$ to each leaf $\Phi_1(\Gamma')$ 
is likewise expressed by \eqref{fifj} 
as a given $C^\omega$ function, specified by the geometric data,
multiplied by the value of $f_1$ at the unique point
at which $\Phi_1(\Gamma')$ intersects $\Phi_1(\bar\Gamma)$.
The latter value of $f_1$ was expressed 
in the preceding paragraph as the product of the scalar $a$
with a specified $C^\omega$ function.
Therefore in a neighborhood of any point of the domain of $f_1$
at which the curves $\Phi_1(\bar\Gamma)$ and $\Phi_1(\Gamma')$ are transverse
to one another, $f_1$ is $C^\omega$.


The assumptions are invariant under permutation of the indices
in $\three$, so corresponding conclusions hold for $f_2$ and for $f_3$.

The case of Lebesgue measurable $\bff$ is treated by applying the same
reasoning to almost every $(\barx,\bars)$.
\end{proof}


The proof of Lemma~\ref{lemma:Comega} has implicitly 
shown that the space of local solutions has dimension less than or equal to $1$:

\begin{lemma} \label{lemma:uniquebff}
In any connected open subset of $\reals^2\times \reals_{\ne 0}$,
equivalence classes of Lebesgue measurable solutions $\bff$ of \eqref{Phisystem}
are unique up to multiplication by constant scalars.
\end{lemma}

\begin{lemma} \label{lemma:sigma}
There exists a unique exponent $\sigma\in\reals$ 
with the following property. Let $\bff$ be
any Lebesgue measurable solution of \eqref{Phisystem} 
in some connected open subset $S$ of $\reals^2\times \reals^+$
that does not vanish identically. Then $f_j$ takes the form
\begin{equation} \label{multform}
f_j(x,s) = s^\sigma F_j(x)  
\ \text{ almost everywhere in $\varphi_j(S)$}
\ \forall\,j\in\three  
\end{equation}
for some $F_j\in C^\omega$.
\end{lemma}

This lemma can be applied equally well to solutions in $\reals^2\times\reals^-$,
with $s^\sigma$ replaced by $|s|^\sigma$ in \eqref{multform}
since whenever $\bff(x,s)$ is a solution, so is $\bff(x,-s)$.
Thus the exponent $\sigma$ is the same for $s<0$ as for $s>0$.

\begin{proof}
If $\bff(x,s)$ is a solution then so is
$\bff(x,rs)$ for any $0\ne r\in\reals$.
Therefore by the uniqueness established in Lemma~\ref{lemma:uniquebff},
$\bff(x,rs)=h(r)\bff(x,s)$ for some function $h$. 
Since $\bff$ does not vanish identically,
this relation implies that $h(r_1r_2)\equiv h(r_1)h(r_2)$
for $r_1,r_2$ in appropriate intervals. This forces
$h$ to take the form $c|r|^\sigma$ in any interval in which it is defined.
\end{proof}

Recalling that $D\varphi_j = V_4\varphi_j$ vanishes nowhere by the transversality hypothesis, 
we may assume that $D\varphi_j(x)>0$ for every $x$ for each index $j$, by
replacing $\varphi_j$ by $-\varphi_j$ if necessary. 
Working in the region in which $s>0$,
and writing $F_j(x)s^\sigma$ in place of $f_j(x,s)$,
the functional equation \eqref{Phisystem} can now be rewritten as
\begin{equation} \label{feqn2}
\sum_{j=1}^3 (D\varphi_j)^\sigma \cdot (F_j\circ\varphi_j)
\cdot \nabla\varphi_j =0.
\end{equation}
Here the first and second factors are real-valued, while the third
is $\reals^2$--valued, and the symbol $\cdot$ denotes multiplication of scalars
with scalars, or with vectors.

\begin{lemma} \label{lemma:sigma_not_0}
If there exists a solution $(F_1,F_2,F_3)$ of \eqref{feqn2}
that does not vanish identically in some nonempty open set in $\reals^2$, 
then the exponent $\sigma$ cannot equal $0$.
\end{lemma}

\begin{proof}
Let $(F_1,F_2,F_3)$ be a solution. Choose
$G_j$ to be an antiderivative  of $F_j$
for each $j\in\three$. If $\sigma=0$ then 
each factor $(D\varphi_j)^\sigma$ is constant, so 
\eqref{feqn2} can be written as
$\nabla\big(\sum_{j=1}^3 (G_j\circ\varphi_j)\big)\equiv 0$
in a nonempty open subset of $\reals^2$.
Thus $\sum_{j=1}^3 (G_j\circ\varphi_j)$ is locally constant.

Choosing $G_4$ to be an appropriate constant, 
$\sum_{j=1}^4 (G_j\circ\varphi_j)$ vanishes identically.
By the main hypothesis, all $G_j$ must then be locally constant.
Therefore their derivatives $F_j$ vanish identically.
\end{proof}

The discussion so far leaves open the question of whether there exists any solution 
besides the trivial solution $\bff\equiv 0$.
The auxiliary hypothesis \ref{auxhyp1} of Theorem~\ref{maintheorem} excludes nonzero
solutions with $\sigma=1$, but not with other exponents.

\subsection{Approximate solutions} 

\begin{lemma} \label{lemma:approxsolns1}
Suppose that there is no nonzero $C^\omega$ solution
of \eqref{Phisystem} in any nonempty open subset of $\reals^3$.
Then there exists $\kappa>0$ with the following property.
Let $I\subset\reals\setminus\{0\}$ be a compact interval.
Let $B\subset\reals^2$ be a ball.
Let $\eps\in(0,1]$.
Let $\bff = (f_j(x,s): j\in\three)$
be an ordered triple of Lebesgue measurable functions.
Suppose that there exists $i\in\three$
for which $|f_i(x,s)|\ge 1$ for every $(x,s)\in B\times I$.
Then
\begin{equation} \label{sublevelbound1}
\big|\{ (x,s)\in B\times I: 
\big| \sum_{j=1}^3 (f_j\circ\Phi_j)(x,s) \,\nabla\varphi_j(x) \big| \le\eps \}\big|
\le C\eps^\kappa.
\end{equation}
\end{lemma}

\begin{proof}
Let $\bff$ satisfy the hypotheses.
It suffices to prove the conclusion in the special case in which
$|f_j(x,s)|\le 2$ for every $j\in\three$ and every $(x,s)\in B\times I$.
Indeed, partition the domain into sets $\scripta_{l,k}$
in which there exists $l$ such that $|f_l(x,s)|\in[2^k,2^{k+1})$
and $|f_j(x,s)|\le |f_l(x,s)|$ for every $j\ne l$.
Apply the special case to $2^{-k}f_j(x,s)$
with $\eps$ replaced by $2^{-k}\eps$
to conclude that $|\scripta_{l,k}| = O(2^{-k}\eps)^\kappa$.
Sum these bounds over all nonnegative integers $k$ and over $l\in\three$
to complete the proof.

Since the hypotheses are invariant under permutation of the indices
in $\three$, we may assume that the index $i$ in the hypotheses is $1$.
Let 
\begin{equation} \scripte = \big\{ (x,s)\in B\times I: 
\big| \sum_{j=1}^3 (f_j\circ\Phi_j)(x,s) \,\nabla\varphi_j(x) \big| \le\eps
\big\}.  \end{equation}
The proof of Lemmas~\ref{lemma:Comega} and \ref{lemma:uniquebff} above,
in combination with reasoning in \cite{triosc}, \cite{sublevel4}, and \cite{CDR},
demonstrates that there exist $C^\omega$ functions $g_j$
that are independent of $\bff$,
a set $\tilde\scripte\subset\scripte$ satisfying
$|\tilde\scripte|\gtrsim|\scripte|^C$,
and scalars $a_j\in[-3,3]$ that may depend on $\bff$, satisfying
\begin{equation} \label{approxbffbyComega}
\big| f_j(\Phi_j(x,s))-a_j g_j(\Phi_j(x,s))
\big| = O(\eps)\qquad \forall\,j\in\three,\ \forall\,(x,s)\in\tilde\scripte.
\end{equation}
See for instance the derivation of relation (6.20) of \cite{sublevel4}
and the discussion in \S4 of \cite{sublevel4} on which that derivation is based.

The family of tuples 
$G_{\ba} = (a_j g_j: j\in\three)$ parametrized by $\ba = (a_1,a_2,a_3)\in
[-3,3]^3$ is a compact $C^\omega$ family of $C^\omega$ functions
in the sense hypothesized in Lemma~11.2 of \cite{triosc}. 
That is, $\ba$ varies over a compact subset of a Euclidean space,
$y =(x,s)$ varies over a compact convex subset of a Euclidean space,
$G_\ba(y)$ is $\complex^n$--valued for some $n$, and
the mapping $(y,\ba)\mapsto G_\ba(y)$ is defined and real analytic in a neighborhood of this domain.
Lemma~11.2 of \cite{triosc} states\footnote{The result in \cite{triosc} is stated and
proved for scalar linear equations, but the same reasoning applies
to finite-dimensional linear systems.} that if such a family contains no exact solution to 
a system such as $\sum_j a_j (g_j\circ\Phi_j) \nabla\varphi_j\equiv 0$
in any nonempty open set, then 
\eqref{sublevelbound1} holds for some exponent $\kappa$ 
for all $G_\ba$ in the family, uniformly in $\ba$.
Since $\tilde\scripte\subset\scripte$, and since the inequality in question
holds on $\scripte$,  
the inequality \eqref{approxbffbyComega} implies that
\begin{equation} 
\big| \sum_{j=1}^3 a_j (g_j\circ\Phi_j)(x,s) \,\nabla\varphi_j(x) \big| =O(\eps)
\ \forall\,(x,s)\in\tilde\scripte.
\end{equation}
Therefore $|\tilde\scripte| = O(\eps^\gamma)$ for some exponent
$\gamma>0$ that depends only on $\Pphi$.
Therefore $|\scripte| = O(\eps^{\gamma/C})$.
\end{proof}

\begin{lemma} \label{lemma:approxsolns2}
Under the hypotheses of Lemma~\ref{lemma:approxsolns1},
there exists $\kappa>0$ such that
for any $r>0$,
\begin{equation} \label{sublevelbound2}
\big|\{ (x,s)\in \reals^2\times [-r,r]:
\big| \sum_{j=1}^3 (f_j\circ\Phi_j)(x,s) \,\nabla\varphi_j(x) \big| \le\eps \}\big|
\le C\eps^\kappa r.
\end{equation}
\end{lemma}

\begin{proof}
To obtain this upper bound for the set of all $(x,s)$
satisfying $\tfrac12 r\le |s|\le r$,
apply Lemma~\ref{lemma:approxsolns1} 
to the functions $(x,s)\mapsto f_j(x,rs)$.
Then apply this partial result with $r$
replaced by $2^{-n}r$ for each $n\in\naturals$, and sum over $n$.
\end{proof}



The next lemma is a variant of Lemma~\ref{lemma:approxsolns2},
concerning the case in which \eqref{Phisystem} does admit a nontrivial solution.
Let $B$ be as above.

\begin{lemma} \label{lemma:approxsolns3}
Suppose that there exist $\bF=(F_j: j\in\three)$,
defined in $B\times\reals^1_{\ne 0}$
and not vanishing identically in any nonempty open set,
and an exponent $\sigma\in\reals$, such that 
every $\reals^3$--valued exact solution $\bff$
of \eqref{Phisystem} in any nonempty connected open subset of $B\times\reals^\pm$
takes the form $f_j(y,s) = b |s|^\sigma F_j(y)$ for some coefficient $b\in\reals$
that depends on $\bff$. 
Then there exist $\kappa,c>0$ and $C_0<\infty$ with the following property.

Let $\bff = (f_j: j\in\three)$
be an ordered triple of Lebesgue measurable functions.
Let $\eps\in(0,1]$. Let $S(\bff,\eps)$ be the sublevel set
\[ S(\bff,\eps) = \big\{(x,s)\in B\times (0,1]:
\big| \sum_{j=1}^3 (f_j\circ\Phi_j)(x,s) \,\nabla\varphi_j(x) \big| \le\eps \big\}.\]
Let $\bbS\subset S(\bff,\eps)$ be measurable.
Then either 
\begin{equation}
|\bbS|\le\eps^c
\end{equation}
or there exist $\bbS'\subset \bbS$ and $b\in\reals$ such that for each $j\in\three$,
\begin{equation} 
\big|\{ (y,s)\in \psi_j(\bbS'):
\big|f_j(y,s) - b |s|^\sigma F_j(y) \big| \le C_0\eps \}\big|
\gtrsim |S(\bff,\eps)|^{\kappa} \ \forall\,\eps>0.
\end{equation}
\end{lemma}

In the case in which no nontrivial solutions of \eqref{Phisystem} exist,
the hypotheses of Lemma~\ref{lemma:approxsolns3} are satisfied with $\bF\equiv 0$, 
but the conclusion is weaker than that of Lemma~\ref{lemma:approxsolns2}.
Lemma~\ref{lemma:approxsolns3} is proved by excerpting portions of
the proofs of Lemmas~\ref{lemma:approxsolns1} and \ref{lemma:approxsolns2}.
\qed

\section{The case in which equation \eqref{Phisystem} has only the trivial solution}
\label{section:nosolutions}

We prove Lemma~\ref{lemma:flatsublevel} in the case in which
every solution of \eqref{Phisystem} in any nonempty open set vanishes identically.
Let $k_j$ be the functions given in the statement of Lemma~\ref{lemma:flatsublevel},
and let $\scripts^*(\bm)$ be as defined at the end of \S\ref{section:flat}.
For each $j\in\three$ define 
\begin{equation} f_j(y,s) = \lambda^{-\tau_0} k_j(m,s)
\ \forall\,y\in I_{m}^*.  \end{equation}
Then whenever $x\in Q_\bm^*$ and $s\in\scripts^*(\bm)$,
\begin{equation} \label{sublevelforf}
\sum_{j\in\three} (f_j\circ\Phi_j)(x,s)\,\nabla_x\varphi_j(x)
= 
\lambda^{-\tau_0} \sum_{j\in\three} k_j(m_j,s)\,\nabla_x\varphi_j(x)
= O(\lambda^{\rho_0-\tau_0}).
\end{equation}
Moreover, since $|k_1(m,s)|\ge \lambda^{\tau_0}$, $|f_1(y,s)|\ge 1$
whenever $y\in\varphi_1(Q_\bm^*)$ and $s\in\scripts^*(\bm)$.

Define
\begin{equation}
\scripta = \big\{ (x,s): 
\big|\sum_{j\in\three} (f_j\circ\Phi_j)(x,s)\,\nabla_x\varphi_j(x)\big|
= O(\lambda^{\rho_0-\tau_0}) \big\}. 
\end{equation}
For any $\bm$ and any $s = O(\lambda^{-\gamma})$,
if $s\in \scripts^*(\bm)$ then for any $x\in Q_\bm^*$,
\begin{equation}
\big|\sum_{j\in\three} (f_j\circ\Phi_j)(x,s)\,\nabla_x\varphi_j(x)\big|
= O(\lambda^{\rho_0-\tau_0}). 
\end{equation}
This is simply the defining property of $\scripts^*(\bm)$,
rewritten in terms of the functions $f_j$.  Thus
\begin{equation}
\bigcup_\bm \big(Q_\bm \times \scripts^*(\bm)\big) \subset\scripta.
\end{equation}

According to Lemma~\ref{lemma:approxsolns2},
since $s$ varies over $[-r,r]$ for $r = C\lambda^{-\gamma}$,
\begin{equation}
|\scripta| = O(\lambda^{-(\tau_0-\rho_0)\kappa} \cdot \lambda^{-\gamma})
\end{equation}
where $\kappa>0$ depends only on $\Pphi$.
Now \begin{equation} 
|\scripts| = \sum_\bm |\scripts^*(\bm)|
\asymp \lambda^{2\gamma} \big| \bigcup_\bm \big(Q_\bm \times \scripts^*(\bm)\big) \big| 
\le \lambda^{2\gamma}|\scripta|
=
O(\lambda^{\gamma -(\tau_0-\rho_0)\kappa}).
\end{equation}
Since $(\tau_0-\rho_0)\kappa>0$,
this completes the proof of Lemma~\ref{lemma:flatsublevel} for this case.
\qed

\section{$0$ and $1$ are the only relevant values of $\sigma$}\label{section:mostsigma}

The analysis developed thus far does not lead to the desired conclusion
if there are many $m$ such that $f_{j,m}$ has the property that 
$f_{j,m,s}(x) = f_{j,m}(x+s)\overline{f_{j,m}}(x)$
has the property that $|\widehat{f_{j,m,s}}(\xi(s))|$
is relatively large for many $s$, for a function $\xi$ that is
well approximated by $ c(j,m)s^\sigma$ for many $m$.
The exponents $\sigma = 0,1$ are excluded by hypotheses of Theorem~\ref{maintheorem}.
In this section we develop a lemma which will subsequently be used to show that for $\sigma\notin\{0,1\}$,
$|\widehat{f_{j,m,s}}(\xi(s))|$
cannot be large for many $s$ for such a function $\xi$. 
This result will be used in \S\ref{section:conclusionofproof} to complete the proof of
Lemma~\ref{lemma:flatsublevel}. 

\begin{lemma} \label{lemma:sigmatriform}
For each $\sigma\in\reals\setminus\{0,1\}$ there exists $\delta_2>0$ with the following property.
Let $\mu\in(0,1)$ and $B\in[0,\infty)$.
Let $r\in[1,\infty)$ be arbitrary, and let $\xi:\reals\to\reals$ 
be Lebesgue measurable and satisfy
\begin{equation} 
\big|\xi(s) - rs^\sigma\big| \le Br^\mu\ \ \forall\,s>0.  
\end{equation}
For any $f,g\in L^\infty(\reals)$ 
supported on some common interval of length $1$, 
\begin{equation}
\int_{s>0} \Big| \int_{\reals}   f(x+s)\,g(x)\,  e^{i \xi(s)x}\,dx \Big| \,ds
\le Cr^{-\delta_2} \norm{f}_\infty\norm{g}_\infty.
\end{equation}
The constant $C$ depends only on $\sigma,\mu,B$.
\end{lemma}

The same conclusion holds for the corresponding integral over $s\in\reals^-$.

The conclusion is false for $\sigma=0$
(consider $f_j(y) = \overline{g_j}(y) = e^{i\xi y}$) and also false for $\sigma=1$ 
(consider $f_j(y) = \overline{g_j}(y) = e^{i\xi y^2}$). 


\begin{proof}
The conclusion can be equivalently stated 
in terms of a trilinear expression
\begin{equation}
I= \iint  f(x+s)\,g(x)\,h(s)\,  e^{i \xi(s)x}\,
\one_\Omega(x,s)
\,dx \,ds,
\end{equation}
where $\Omega$ is the set of all $(x,s)$ such that
$x\in J$, $x+s\in J$, and $s\in\reals^+$
for a certain interval $J\subset\reals$ of length $1$. 
This set $\Omega$ is convex.
The conclusion is then that
\begin{equation} |I|\le Cr^{-\delta_2} \norm{f}_\infty\norm{g}_\infty\norm{h}_\infty \end{equation}
for any three functions $f,g,h$. 

Assume that $\norm{f}_\infty\le 1$ and likewise for $g,h$.
By Cauchy-Schwarz, $|I|^2$ is majorized by
\begin{align*}
C & \iiint f(x+s') \bar f(x+s)  h(s')\bar h(s) 
\,e^{i[\xi(s')-\xi(s)]x}
\,\one_{\Omega}(x,s)\one_\Omega(x,s')
\,dx\,ds\,ds'
\\
&= C \iiint f(x+s+t)\bar f(x+s) h(s+t)\bar h(s) 
\,e^{i[\xi(s+t)-\xi(s)]x}
\,\one_{\Omega}(x,s)\one_\Omega(x,s+t)
\,dx\,ds\,dt
\\
&= C \int\Big(\iint F_t(x+s) H_t(s)
\,e^{i[\xi(s+t)-\xi(s)]x}
\,\one_{\Omega}(x,s)\one_\Omega(x,s+t)
\,dx\,ds\Big)\,dt
\\
&= C \int\Big(\iint F_t(y) \tilde H_t(s)
\,e^{i[\xi(s+t)-\xi(s)]y}
\,\one_{\Omega}(y-s,s)\one_\Omega(y-s,s+t)
\,dy\,ds\Big)\,dt
\end{align*}
with $F_t(x) = f(x+t)\bar f(x)$, $H_t(s) = h(s+t)\bar h(s)$,
and $\tilde H_t(s) = H_t(s) e^{-i[\xi(s+t)-\xi(s)]s}$.
The outer integral extends only over some interval of length $2$.

A second application of Cauchy-Schwarz results in majorization
of the square of the absolute value of the inner integral by
\begin{equation}
C \iint \tilde H_t(u) \overline{\tilde H_t(v)}
\Big(\int e^{i[\xi(u+t)-\xi(u)-\xi(v+t)+\xi(v)]y}
\,\one_{\Omega'}(y,u,v,t)
\,dy\Big)
\,du\,dv
\end{equation}
for each $t$,
with
\[ \one_{\Omega'}(y,u,v,t)
= \,\one_{\Omega}(y-u,u)\one_\Omega(y-u,u+t)
\,\one_{\Omega}(y-v,v)\one_\Omega(y-v,v+t).  \]
The set $\Omega'\subset\reals^4$, whose indicator function is denoted here
by $\one_{\Omega'}$, is convex.

Define $\Omega''$ to be the set of all $(u,v,t)$
for which there exists $y$ such that $(y,u,v,t)\in\Omega'$.
$\Omega''$ is convex, and is contained in a ball of radius $O(1)$,
centered at the origin, in $\reals^3$.
Replacing the inner integral by its absolute value 
and invoking the hypothesis $\norm{h}_{L^\infty}\le 1$,
we conclude that $|I|^4$ is majorized by
\begin{multline}
C \iiint 
\Big| \int e^{i[\xi(u+t)-\xi(u)-\xi(v+t)+\xi(v)]y}
\,\one_{\Omega'}(y,u,v,t)
\,dy\Big|
\,dt\,du\,dv
\\ \lesssim
\iiint_{\Omega''} 
\big(1+\big|\xi(u+t)-\xi(u)-\xi(v+t)+\xi(v)\big|)^{-1}
\,dt\,du\,dv.
\end{multline}
The inequality holds because for each $(u,v,t)\in\Omega''$,
$y\mapsto\one_{\Omega''}(y,u,v,t)$ is the indicator function
of an interval of length $O(1)$.

Each quantity $t,u,v$ is confined to an interval of length $O(1)$
centered at the origin in $\reals$,
and there is an implicit restriction to the region
in which $u,v,u+t,v+t$ are all positive.

The structural hypothesis $\xi(s) = rs^\sigma+ O(r^\mu)$ 
and the restriction $|s| = O(1)$ implicit in the hypotheses give
\begin{equation}
r^{-1} \big|\,\xi(u+t)-\xi(u)-\xi(v+t)+\xi(v)\,\big| 
= \big|\,(u+t)^\sigma-(u)^\sigma-(v+t)^\sigma+v^\sigma\,\big| + O(r^{\mu-1}).
\end{equation}

For any $\sigma\notin\{0,1\}$
there exists $\kappa>0$ such that for any $\eps>0$,
\begin{equation}
\big|\big\{ (u,v,t)\in \Omega'': 
|(u+t)^\sigma-u^\sigma-(v+t)^\sigma+v^\sigma| <\eps \big\} \big|
= O(\eps^\kappa).
\end{equation}
This can be shown by fixing $(v,t)$ and allowing $u$ to vary
over the interval $\{u: (u,v,t)\in\Omega''\}$.
Details of the elementary proof are left to the reader.

Choose $\eps = r^{\mu-1}$. 
Then for any $(u,v,t)$ satisfying
$r^{-1} |\xi(u+t)-\xi(u)-\xi(v+t)+\xi(v)| <\eps$,
one also has $|(u+t)^\sigma-u^\sigma-(v+t)^\sigma+v^\sigma| \le C\eps$.

Therefore 
\begin{equation} \big|\big\{(u,v,t)\in\Omega'':
|\xi(u+t)-\xi(u)-\xi(v+t)+\xi(v)| \le r^{\mu} \big\}\big| 
= O(r^{-(1-\mu)\kappa}), \end{equation}
and consequently
\begin{equation}
\iiint 
\big(1+\big|\xi(u+t)-\xi(u)-\xi(v+t)+\xi(v)\big|)^{-1}
\,dt\,du\,dv
= O\big(r^{-(1-\mu)\kappa} + r^{-\mu}\big). 
\end{equation}
\end{proof}

\section{Conclusion of proof of Lemma~\ref{lemma:flatsublevel}} \label{section:conclusionofproof}

We now complete the proof of Lemma~\ref{lemma:flatsublevel},
by proving the upper bound \eqref{ineq:mainsublevelrestated},
that is, $\sum_\bm |\scripts^*(\bm)| = O(\lambda^{\gamma-\varrho})$,
in the case in which there does exist a real analytic solution $\bF$,
not identically zero, of \eqref{Phisystem} in some nonempty open set.
The case in which every real analytic exact solution 
vanishes identically, has already been treated in \S\ref{section:nosolutions}.

We have shown that the exact solutions in any connected open
set form a one-dimensional vector space,
and take the form $s^\sigma \bF(y)$, with 
$\bF\in C^\omega$;  
$\sigma\in\reals$ is uniquely determined by $(\varphi_j: j\in\four)$.
Choose and fix such a function $\bF$ that does not vanish identically.
By auxiliary hypothesis~\ref{auxhyp1}, $\sigma\ne 1$.
By Lemma~\ref{lemma:sigma_not_0}, $\sigma\ne 0$.

For $j\in\three$ define $K_j$ by
\begin{equation}
K_j(y,s) = \lambda^{-\tau_0} k_j(m,s)
\ \text{ for $y\in I_m^*$}
\end{equation}
and for $s = O(\lambda^{-\gamma})$,
where $m$ is chosen to have the appropriate residue modulo $3$
for the index $j$.
Then $|K_1(y,s)|\ge 1$ for every $(y,s)$.

Define $S(\bK)$ to be the set of all $(x,s)\in B\times [-C\lambda^{-\gamma},C\lambda^{-\gamma}]$
such that there exists $\bm$ with $x$ in the support of $\zeta_\bm$
satisfying $s\in \scripts^*(\bm)$. The desired bound
$\sum_\bm |\scripts^*(\bm)| = O(\lambda^{\gamma-\varrho})$
is equivalent to
\begin{equation} \label{SbKbound} \lambda^\gamma |S(\bK)| = O(\lambda^{-\varrho}).\end{equation}

By Lemma~\ref{lemma:approxsolns3}, either $\lambda^\gamma|S(\bK)|\le \lambda^{-c}$
for a certain constant $c>0$ 
--- in which case the proof of Lemma~\ref{lemma:flatsublevel} is complete ---
or there exist a measurable set $S\subset S(\bK)$ 
and a scalar $b\in\reals$ that satisfy
\begin{equation} \label{approxK}
|K_1(y,s)-b|s|^\sigma F_1(y)| \lesssim \lambda^{\rho-\tau_0}
\ \text{ for all $(y,s)\in S$}
\end{equation}
and
\begin{equation} \lambda^\gamma |S| \gtrsim (\lambda^\gamma |S(\bK)|^C.  \end{equation}
Thus in order to complete the proof of Lemma~\ref{lemma:flatsublevel}, 
it suffices to show that in this second case,
\begin{equation} \lambda^\gamma|S| = O(\lambda^{-\varrho}) 
\ \text{ for some $\varrho>0$.} \end{equation}

Let $\delta_3>0$ be another small parameter satisfying requirements
to be specified below.
We may assume that $|s|\ge\lambda^{-\gamma-\delta_3}$
for every $(y,s)\in S$.
Indeed, if the set $S'$ of all $(y,s)\in S$ satisfying this
inequality had Lebesgue measure $\le\tfrac12|S|$,
then $|S|\le 2|S\setminus S'| \lesssim \lambda^{-\gamma-\delta_3}$,
and the proof would be complete.
Thus we may assume that $|S'|\ge \tfrac12|S|$.

We may assume that $|S'|>0$, and in particular, that $S'$ is nonempty.
Choose an arbitrary point $(\bary,\bars)\in S'$.
Since $|K_1(\bary,\bars)|\ge 1$, $\big| b|\bars|^\sigma F_1(\bary) \big|
\ge 1-O(\lambda^{\rho-\tau_0})\ge \tfrac12$, since $\tau_0>\rho$
and we may assume $\lambda$ to be large.
The function $\bF$ is fixed, and is bounded above.
Since $\lambda^{-\gamma-\delta_3}\le |\bars| \le\lambda^{-\gamma}$
it follows that 
\begin{equation}
\left\{ \begin{aligned}
&|b|\lambda^{-\gamma\sigma} \gtrsim 1 \ \text{ if $\sigma>0$,}
\\
&|b|\lambda^{-(\gamma+\delta_3)\sigma} \gtrsim 1\ \text{ if $\sigma<0$.}
\end{aligned} \right. \end{equation}
Thus 
\begin{equation} \label{bsigmabound}
|b|\lambda^{-\sigma\gamma} \gtrsim \lambda^{-|\sigma|\delta_3}.
\end{equation}

The function $F_1$ is real analytic and does not vanish identically,
for if $F_1\equiv 0$ then $|K_1(y,s)-0|= O(\lambda^{\rho-\tau_0})$
for all $(y,s)\in S$, contradicting the condition $|K_1(y,s)|\ge 1$
unless $S = \emptyset$.
Therefore there exist $c>0$  and $C<\infty$ such that 
\begin{equation}
\big|\big\{y\in \varphi_1(B): |F_1(y)|\le\eps\big\}\big|
\le C\eps^c\ \forall\,\eps>0.
\end{equation}
These quantities $c,C$ depend only on $\Pphi$.
Therefore 
\begin{equation}
\big|\big\{(x,s)\in B\times [-C\lambda^{-\gamma},C\lambda^{-\gamma}]: 
|F_1(\varphi_1(x))| \le \lambda^{-\delta_4}\big\}\big|
= O(\lambda^{-c\delta_4}\lambda^{-\gamma}).
\end{equation}

Thus it suffices to establish an upper bound of the same type
for the measure of the set of all $(x,s)\in S'$
that satisfy the supplementary inequality $|F_1(\varphi_1)(x)|>\lambda^{-\delta_4}$.

For any such $x$, and any interacting $\bm$ for which $x$ belongs to the support of $\zeta_\bm$,
for any other $x'$ in the support of $\zeta_\bm$,
\[|F_1(\varphi_1)(x')|\ge |F_1(\varphi_1)(x)|-C\lambda^{-\gamma}
\ge \lambda^{-\delta_4}-\lambda^{-\gamma}\gtrsim \lambda^{-\delta_4}\] 
since $|x-x'| = O(\lambda^{-\gamma})$ and $\delta_4<\gamma$.

Thus it suffices to show that for any interacting $\bm$
with the supplementary property that $|F_1(\varphi_1(x'))|\gtrsim \lambda^{-\delta_4}$
for every $x'$ in the support of $\bm$, the following holds:
For any $x$ that is in the support of $\bm$,
$S'(x) = \{s\in\reals: (x,s)\in S' \}$ satisfies
$|S'(x)| = O(\lambda^{-\gamma-\varrho})$.

Consider any such $x$.
The coefficients in the local Fourier series \eqref{localFourier} are by definition
\begin{align*} a_{j,m,s,k_j(m,s)} 
= \tfrac12 \lambda^{\gamma} \int \eta_m(y) f_j(y+s)\overline{f_j(y)}\,
e^{-i \pi\lambda^\gamma k_j(m,s)y} \,dy.
\end{align*}
By the inequality \eqref{extrainfo}, which is one of the hypotheses 
of Lemma~\ref{lemma:flatsublevel}, these satisfy 
\[|a_{1,m_1,s,k_1(m_1,s)}|\gtrsim\lambda^{-\delta_1}.\] 
This lower bound can be exploited to obtain an upper bound for $|S'(x)|$.
To begin,
\begin{multline} \label{Fcoefftolowerbound}
\int_{S'(x)} 
\big|  \int \eta_{m_1}(y) f_1(y+s)\overline{f_1(y)} e^{-i\pi\lambda^\gamma
k_1(m_1,s) y}\,dy\big|\,ds
\\ \gtrsim \int_{S'(x)}
\lambda^{-\gamma} |a_{1,m_1,s,k_1(m_1,s)}|\,ds \gtrsim \lambda^{-\gamma-\delta_1} |S'(x)|.
\end{multline}

To continue,
an upper bound for the left-hand side of \eqref{Fcoefftolowerbound} can be derived from 
Lemma~\ref{lemma:sigmatriform}  after a change of variables. Indeed,
set $\tilde s = \lambda^\gamma s$ and $\tilde x = \lambda^\gamma x$.
By \eqref{approxK},
the phase $\pi\lambda^\gamma k_1(m_1,s)\,y$ is equal to $\xi(\tilde s)\,\tilde y$ with
\begin{equation}
\xi(t) = \pi \lambda^{\tau_0} K_1(m_1,\lambda^{-\gamma} t)
	= \pi \lambda^{\tau_0} b\lambda^{-\gamma\sigma} |t|^\sigma F_1(\varphi_1(x))
+O(\lambda^\rho).  \end{equation}


In these coordinates, we aim to apply Lemma~\ref{lemma:sigmatriform} with the parameter
\begin{equation} r = \pi \lambda^{\tau_0} |b| \lambda^{-\gamma\sigma} F_1(\varphi_1(x)).  \end{equation}
The lemma is applicable provided that $r\ge 1$ and $\lambda^\rho\le r^\mu$ for some $\mu<1$;
we next verify that these inequalities are satisfied.

Since $|b\lambda^{-\gamma\sigma}| \gtrsim \lambda^{-|\sigma|\delta_3}$
by \eqref{bsigmabound},
\begin{equation} r\gtrsim \lambda^{\tau_0-C\delta_3-\delta_4} \end{equation}
and it follows that both
$r\ge 1$ and $\lambda^\rho\le r^{1/2}$ 
for all sufficiently large $\lambda$,
provided that $\rho,\delta_3,\delta_4$ are chosen to be sufficiently
small relative to $\tau_0$.
We require that the parameters be chosen so that
\begin{equation}
C\delta_3+\delta_4 < \tfrac12\tau_0
\ \text{ and } \rho < \tfrac14\tau_0.
\end{equation}
Then Lemma~\ref{lemma:sigmatriform} applies, and yields the upper bound
\begin{equation} \label{Fcoefftoupperbound}
	\begin{aligned}
\int_{S'(x)} 
		\big|  \int \eta_\bm(y) & f_1(y+s)\overline{f_1(y)} e^{i\pi\lambda^\gamma
k_1(m_1,s) y}\,dy\big|\,ds
\\ & \le
\int_{|s| = O(\lambda^{-\gamma})} 
 \big|  \int \eta_\bm(y)  f_1(y+s)\overline{f_1(y)} e^{i\pi\lambda^\gamma
k_1(m_1,s) y}\,dy\big|\,ds
\\ & \lesssim 
\lambda^{-2\gamma} \cdot (\lambda^{\tau_0-C\delta_3-\delta_4})^{-\delta_2}.
\end{aligned}
\end{equation}
Together with \eqref{Fcoefftolowerbound}, 
\eqref{Fcoefftoupperbound} yields a lower bound for $|S'(x)|$:
\begin{equation}
|S'(x)| \lesssim \lambda^{-\gamma+\delta_1} \, \lambda^{-\tau_0\delta_2/4}.
\end{equation}
We require that
\begin{equation} \delta_3 < \delta_1 < \tfrac18 \tau_0\delta_2.  \end{equation}

Recall that $\sigma$ is uniquely determined by $(\varphi_j: j\in\four)$,
while $\delta_2$, which arose in Lemma~\ref{lemma:sigmatriform},
depends only on $\sigma$, and $\delta_1$ is assumed to be less than or equal to
$\delta_0$, which has not yet been specified. Therefore these requirements 
can be satisfied by first choosing $\delta_0<\tfrac18\tau_0\delta_2$,
then choosing $\delta_3,\delta_4$ sufficiently small relative to $\delta_0,\delta_2$.
Thus we have shown that if the parameters $\delta_0,\delta_3,\delta_4,\rho$
are chosen to satisfy the indicated constraints then uniformly for all $x\in B$
under consideration,
\begin{equation} |S'(x)| = O(\lambda^{-\gamma-\delta_5}) \end{equation}
for a certain exponent $\delta_5>0$.
This completes the proof of Lemma~\ref{lemma:flatsublevel}.

\section{Conclusion of proof} \label{section:final_conclusion}

Assume that $\norm{\bff}_{L^\infty} = O(1)$.
Let $\delta_1^*,\delta_2^*,\delta_3^*,\delta_4^*\in(0,\delta_0]$ be small positive quantities, 
to be chosen below.
For each index $j$, decompose $f_j = f_{j,\sharp} + f_{j,\flat}$ 
by applying Lemma~\ref{lemma:sharpflat} 
in the manner indicated in the discussion following its statement, with parameter $\delta = \delta_j^*$.

The summands $f_{j,\sharp}$ satisfy
\begin{equation}
\norm{f_{j,\sharp}}_{L^\infty} = O( \lambda^{a_j})
\end{equation}
where $a_j= \delta_j^*\tau_0/2$ tends to $0$ as $\delta_j^*\to 0$.

Express
\begin{equation}
\scriptt(\bff) 
= \scriptt(f_1,f_2,f_3,f_{4,\sharp})
+ \scriptt(f_1,f_2,f_3,f_{4,\flat}).
\end{equation}
By Proposition~\ref{prop:flat},
\begin{equation}
|\scriptt(f_1,f_2,f_3,f_{4,\flat})| = O(\lambda^{-\tau_4})
\end{equation}
for a certain positive constant $\tau_4>0$ that depends on
the choice of $\delta_4^*$. 
Express
\begin{equation}
\scriptt(f_1,f_2,f_3,f_{4,\sharp})
= \scriptt(f_1,f_2,f_{3,\sharp},f_{4,\sharp})
+ \scriptt(f_1,f_2,f_{3,\flat},f_{4,\sharp}).
\end{equation}
By Proposition~\ref{prop:flat},
\begin{equation}
|\scriptt(f_1,f_2,f_{3,\flat},f_{4,\sharp})|
= O(\lambda^{-\tau_3}\norm{f_{4,\sharp}}_{L^\infty}) = O(\lambda^{-\tau_3 + a_4} ),
\end{equation}
for a certain $\tau_3>0$ that depends on $\delta_3^*$.
Continuing in this way gives
\begin{multline} \label{secondtolast}
|\scriptt(\bff)|
=O(\lambda^{-\tau_4})
+ O(\lambda^{-\tau_3+a_4})
+ O(\lambda^{-\tau_2+a_3+a_4})
\\
+ O(\lambda^{-\tau_1+a_2 + a_3+a_4})
+ |\scriptt(f_{1,\sharp},f_{2,\sharp},f_{3,\sharp},f_{4,\sharp})|,
\end{multline}
with each $\tau_j$ being positive, and depending on $\delta_j^*$.
Finally, since each function $f_{j,\sharp}$ is a sum of $O(\lambda^{a_j})$
functions $g_j$ of the type that appears in Proposition~\ref{prop:sharp},
that Proposition provides a majorization
\begin{equation} \label{last}
\scriptt(f_{1,\sharp},f_{2,\sharp},f_{3,\sharp},f_{4,\sharp})
= O\big(\lambda^{-\tau} \prod_{j=1}^4 \norm{f_{j,\sharp}}_{L^\infty}\big)
= O\big(\lambda^{-\tau+ a_1+a_2+a_3+a_4}\big)
\end{equation}
for a certain exponent $\tau>0$ that does not depend on the quantities $\delta_j^*$.

Choose the parameters $\delta_j^*$ to be sufficiently small to ensure
that $a_j\le \tau/8$ for each $j$, and choose them by ascending induction on $j$
so that each $\delta_n^*$ is sufficiently small as a function of $\delta_{n-1}^*$
to ensure that 
each term on the right-hand side of \eqref{secondtolast} is of the form $O(\lambda^{-c})$
for some $c>0$. This completes the proof of Theorem~\ref{maintheorem}.

\section{A sublevel set inequality}

For any four-tuple $\bff$ of measurable functions $f_j:\varphi_j(B)\to\complex$
define
\begin{equation}
	S(\bff,\eps) = \big\{x\in B: |\sum_{j=1}^4 (f_j\circ\varphi_j)(x) | <\eps\big\}
\end{equation}
and
\begin{equation}
	|(\bff\circ\bPhi)(x)| = \sum_{j=1}^4 |(f_j\circ\varphi_j)(x)|.
\end{equation}

\begin{corollary}
Let $(\ba,\bPhi)$ satisfy the hypotheses of Theorem~\ref{maintheorem}.
There exist $C<\infty$ and $\tau>0$ such that for any Lebesgue measurable $\bff$
and any $\eps>0$,
\begin{equation}
\big|\big\{
	x\in S(\bff,\eps): |(\bff\circ\bPhi)(x)|\ge 1 \big\}\big|
\le C\eps^\tau.
\end{equation}
\end{corollary}

This follows from Theorem~\ref{maintheorem} by the same reasoning
as developed for the case of three summands in \S18 of \cite{triosc}.


\end{document}